\documentclass{amsart}
\usepackage{amsfonts,amssymb}
\usepackage{enumerate,graphicx}
\usepackage{caption}
\usepackage{subcaption}
\usepackage[square,numbers]{natbib}
\usepackage{mathrsfs,bbm}
\usepackage{tikz,tikzscale,tikz-cd}
\usepackage{multirow,xparse,fp}
\usepackage{environ}
\usepackage{amstext}
\usepackage{graphicx} %[H]
\usepackage{float} %[H]
\usepackage{hyperref}
\usepackage{diagbox}
 \usepackage[pagewise]{lineno}%\linenumbers
\usepackage{comment}
\usepackage{pgfplots}
\pgfplotsset{compat=1.15}
\usepackage{mathrsfs}
\usetikzlibrary{arrows}

\newtheorem{theorem}{Theorem}[section]

\newtheorem{lemma}[theorem]{Lemma}
\theoremstyle{definition}

\newtheorem{definition}[theorem]{Definition}
\newtheorem{example}[theorem]{Example}

\allowdisplaybreaks[4]%align可跨頁

\numberwithin{equation}{section}%讓第二章節第一個方程式變成編號(2.1)

\begin{document}
\definecolor{ffqqqq}{rgb}{1,0,0}
\definecolor{qqqqff}{rgb}{0,0,1}

\title{Limit sets, internal chain transitivity and orbital shadowing of tree-shifts defined on Markov-Cayley trees}

\author[Jung-Chao Ban]{Jung-Chao Ban}
\address[Jung-Chao Ban]{Department of Mathematical Sciences, National Chengchi University, Taipei 11605, Taiwan, ROC.}
\address{Math. Division, National Center for Theoretical Science, National Taiwan University, Taipei 10617, Taiwan. ROC.}
\email{jcban@nccu.edu.tw}

\author[Nai-Zhu Huang]{Nai-Zhu Huang}
\address[Nai-Zhu Huang]{Department of Mathematical Sciences, National Chengchi University, Taipei 11605, Taiwan, ROC.}
\email{naizhu7@gmail.com}

\author[Guan-Yu Lai]{Guan-Yu Lai}
\address[Guan-Yu Lai]{Department of Mathematical Sciences, National Chengchi University, Taipei 11605, Taiwan, ROC.}
\email{gylai@nccu.edu.tw}

%\thanks{Ban is partially supported by the National Science and Technology Council, ROC (Contract NSTC 111-2115-M-004-005-MY3). Lai is partially supported by the National Science and Technology Council, ROC (Contract NSTC 111-2811-M-004-002-MY2).}

%\date{}

%\baselineskip=1.2\baselineskip

% -------------------------------------------------------------
\begin{abstract}  
In this paper, we introduce the concepts of $\omega$-limit sets and pseudo orbits for a tree-shift defined on a Markov-Cayley tree, extending the results of tree-shifts defined on $d$-trees [5,6]. Firstly, we establish the relationships between $\omega$-limit sets and we introduce a modified definition of $\omega$-limit set based on complete prefix sets (Theorems 1.4 and 1.9). Secondly, we introduce the concept of projected pseudo orbits and investigate the concept of the shadowing property (Theorems 1.12 and 1.14).
\end{abstract}

\maketitle

% -------------------------------------------------------------
%\tableofcontents

\section{Introduction}

\subsection{Motivations}

The objective of this article is to study the $\omega $-limit set and the
shadowing properties for the shift spaces defined on Markov-Cayley trees.
Before we present the main results of our investigation, we highlight the
motivation behind this study.

\textbf{1}. ($\omega $-limit set) If $(X,T)$ is a dynamical system, the $\omega$-limit set of $X$ is undoubtedly the most important set to capture the long-term behavior, recurrence properties and the existence of attractors of the system. In this article, we will be focusing on the $\omega$-limit set of the shift space defined on multidimensional lattices.
In \cite{souza2012limit}, Souza first provided the concepts of the $\omega$-limit set of $G$ actions where $G$ is a group or a monoid. Motivated by the work of Souza, Meddaugh and Raines \cite{meddaugh2014structure} introduced the concepts of an $\omega$-limit set of shifts defined on $\mathbb{Z}^{d}$ and studied the fundamental properties of an $\omega$-limit set and
established the shadowing properties. Later, Binder \cite{binder2019limitPH}, Binder and Meddaugh \cite{binder2019limit} give various types of definitions
for the $\omega$-limit sets of a shift space defined on the free semi-group\footnote{That is, the conventional $d$ tree.} with $d$ generators (briefly, we call it the \emph{shift on the }$d$ \emph{tree}), namely, the $\omega$-limit sets $\omega^{d}(t)$, $\omega_{p}^{d}(t)$, $\omega _{F_{p}}^{d}(t)$ and $\omega_{CPS}^{d}(t)$ where $t$ is an element of the given shift space. The authors identify the connections between these $\omega $-limit sets and
develop some fundamental properties of them. However, we found out that not all the definitions of the $\omega $-limit sets for shifts on the $d$ tree are valid if the underlined lattice, i.e., the $d$ tree, is changed to the
Markov-Cayley tree. Thus, we present a modified definition for the $\omega $-limit set $\omega_{CPS}^{M}(t)$ of a shift on the Markov-Cayley tree $T_{M}$ and investigate the connections among these $\omega $-limit sets (Theorem \ref{Thm: 1}). As it turns out, the outcomes of the shift on the Markov-Cayley tree are vastly different from those on the $d$ tree (Example \ref{ex2}).

\textbf{2}. (Shadowing properties) In a dynamical system, the shadowing property stipulates that any approximate orbit can be traced by a real one. This is significant because rounding errors are common problems in floating-point calculations. We will not attempt to review the extensive literature here, referring only to (\cite{barwell2012shadowing,brian2015chain,bruin2022topological, good2018orbital,good2020shadowing,kulczycki2014almost,kwietniak2012note, li2013shadowing,
meddaugh2013shadowing,
oprocha2008shadowing, pilyugin2006shadowing,pilyugin2003shadowing}) for background and references. In addition, for a deeper discussion of the shadowing properties for shifts on multi-dimensional
lattices, for instance, the $\mathbb{Z}^{d}$, $\mathbb{N}^{d}$ with $d\geq 1$ or some abelian groups, we refer the reader to \cite{oprocha2008shadowing,pilyugin2003shadowing}.  

We confine our discussion to the shadowing property for the shift on the $d$ tree or Markov-Cayley tree based on the above discussion. In \cite{binder2019limit, bucki2024stability}, the authors give the definitions for the (asymptotically) shadowing property for the shift on $d$ tree, and show that a shift on $d$ tree is subshift of finite type (SFT) if and only if the
shift has the shadowing property. This extends the previously well-known result of shifts defined on $\mathbb{N}$ to those defined on the $d$ tree. Furthermore, it is proved that an $m$-step SFT on the $d$ tree has the asymptotically $2^{-(m+1)}$-shadowing property \cite{binder2019limit}. We offer modified definitions
of the (asymptotically) shadowing properties for shifts on Markov-Cayley trees due to the same reason mentioned at the above paragraph. SFTs are characterized as those with the shadowing property based on the modified definition of the shadowing property (Theorem \ref{Thm: 2}). In addition, we
obtain that an $m$-step SFT on the Markov-Cayley tree has an asymptotically $2^{-(m+1)}$-shadowing property. It is worth noting the previous results for shifts on the $d$ tree are extended to a wider range of multi-dimensional shifts because the $d$ tree is a special type of Markov-Cayley tree
(Theorem \ref{Thm: 3}).

We stress that the shifts on Markov-Cayley trees are not invariant, which is a major obstacle to this investigation. This is due to the fundamental difference between the $d$ tree and Markov-Cayley trees. The Markov-Cayley tree is not a type of an abelian group, and the method of shifts on $\mathbb{Z}^{d}$ is not applicable to this class of shifts. In the following section, the formal definitions of the shift on Markov-Cayley trees are introduced.

\subsection{Tree-shifts on Markov-Cayley trees}
Let $\Sigma=\{g_{1},\ldots ,g_{d}\}$ be a finite set. Suppose $T$ is an infinite, locally finite, connected graph with generator set $\Sigma$, without loops and with a distinguished point $\epsilon $. The graph $T$ is also called a \emph{tree}. Let $T$ be a tree and $g\in T$. The \emph{follower set} $F_{T}(g)$ of $g$ is defined as 
\[
F_{T}(g)=\{h\in T:gh\in T\}.
\]
The focus of this article is on a wide range of trees, specifically, Markov-Cayley trees as we define them below. Let $M$ be a $d\times d$ 0-1 matrix indexed by $\Sigma $. The associated \emph{Markov-Cayley tree }$T_{M}$ is defined by 
\[
T_{M}=\{\epsilon \}\cup \Sigma \cup \bigcup\limits_{n=2}^{\infty }\left\{
g_{i_{1}}g_{i_{2}}\cdots g_{i_{n}}:M(g_{i_{j}},g_{i_{j+1}})=1,~\forall
1\leq j\leq n-1\right\} \text{.} 
\]
If $M$ is a $d\times d$ full matrix, i.e., all entries of $M$ are $1^{\prime }$s, then denote $T_{M}$ briefly by $T_{d}$. Let $T$ be a tree, and we denote by 
\[
\mathcal{I=I}_{T}=\{F_{T}(g):g\in T\}
\]
the \emph{family of follower sets of }$T$ and assume that $\mathcal{I}=\{\eta _{1},\ldots ,\eta _{\left\vert \mathcal{I}\right\vert }\}$ with $\left\vert \mathcal{I}\right\vert <\infty $ throughout this article. For $\eta \in \mathcal{I}$, define 
\[
T^{(\eta )}=\{g\in T:F_{T}(g)=\eta \}\text{, }\Delta _{m}^{(\eta )}=\{g\in\eta :0\leq \left\vert g-\epsilon \right\vert \leq m\},
\]
where $|g-\epsilon|:=n$ if $g=g_{i_1}g_{i_2}\cdots g_{i_n}$, and $|g-\epsilon|:=0$ if $g=\epsilon$.

\begin{example}
 For any $d\times d$ 0-1 matrix $M$, we have $\mathcal{I}_{T_M} \subseteq \{T_M, F_{T_M}(g_1), ..., F_{T_M}(g_d)\}$.
\begin{enumerate}
    \item  Let $M$ be a $2\times 2$ full one matrix, $T_M=T_2$ (a conventional 2-tree). Since $F_{T_2}(g_1)=F_{T_2}(g_2)=T_2$, we have $\mathcal{I}=\{\eta\}$, where $\eta=T_2$. Then, $T^{(\eta)}=T_2$.

 \begin{figure}[H]
     \centering
     \begin{subfigure}[b]{0.45\textwidth}
         \centering
              \begin{tikzpicture}[line cap=round,line join=round,>=triangle 45,x=0.6cm,y=0.7cm]
\clip(-2,8.5) rectangle (6.5,11.5);
\draw [line width=0.5pt] (2.25,11)-- (0.25,10.5);
\draw [line width=0.5pt] (2.25,11)-- (4.25,10.5);
\draw [line width=0.5pt] (0.25,10.5)-- (-0.75,10);
\draw [line width=0.5pt] (0.25,10.5)-- (1.25,10);
\draw [line width=0.5pt] (4.25,10.5)-- (3.25,10);
\draw [line width=0.5pt] (4.25,10.5)-- (5.25,10);
\draw [line width=0.5pt] (-0.75,10)-- (-1.25,9.5);
\draw [line width=0.5pt] (-0.75,10)-- (-0.25,9.5);
\draw [line width=0.5pt] (1.25,10)-- (0.75,9.5);
\draw [line width=0.5pt] (1.25,10)-- (1.75,9.5);
\draw [line width=0.5pt] (3.25,10)-- (2.75,9.5);
\draw [line width=0.5pt] (3.25,10)-- (3.75,9.5);
\draw [line width=0.5pt] (5.25,10)-- (4.75,9.5);
\draw [line width=0.5pt] (5.25,10)-- (5.75,9.5);
\draw [line width=0.5pt] (-1.25,9.5)-- (-1.5,9);
\draw [line width=0.5pt] (-1.25,9.5)-- (-1,9);
\draw [line width=0.5pt] (-0.25,9.5)-- (-0.5,9);
\draw [line width=0.5pt] (-0.25,9.5)-- (0,9);
\draw [line width=0.5pt] (0.75,9.5)-- (0.5,9);
\draw [line width=0.5pt] (0.75,9.5)-- (1,9);
\draw [line width=0.5pt] (1.75,9.5)-- (1.5,9);
\draw [line width=0.5pt] (1.75,9.5)-- (2,9);
\draw [line width=0.5pt] (2.75,9.5)-- (2.5,9);
\draw [line width=0.5pt] (2.75,9.5)-- (3,9);
\draw [line width=0.5pt] (3.75,9.5)-- (3.5,9);
\draw [line width=0.5pt] (3.75,9.5)-- (4,9);
\draw [line width=0.5pt] (4.75,9.5)-- (4.5,9);
\draw [line width=0.5pt] (4.75,9.5)-- (5,9);
\draw [line width=0.5pt] (5.75,9.5)-- (5.5,9);
\draw [line width=0.5pt] (5.75,9.5)-- (6,9);
\draw [color=black] (2.25,11) node[] {\fontsize{8 pt}{0pt}\selectfont{$\bullet$}};
\draw [color=black] (0.25,10.5) node[] {\fontsize{8 pt}{0pt}\selectfont{$\bullet$}};
\draw [color=black] (4.25,10.5) node[] {\fontsize{8 pt}{0pt}\selectfont{$\bullet$}};
\draw [color=black] (-0.75,10) node[] {\fontsize{8 pt}{0pt}\selectfont{$\bullet$}};
\draw [color=black] (1.25,10) node[] {\fontsize{8 pt}{0pt}\selectfont{$\bullet$}};
\draw [color=black] (3.25,10) node[] {\fontsize{8 pt}{0pt}\selectfont{$\bullet$}};
\draw [color=black] (5.25,10) node[] {\fontsize{8 pt}{0pt}\selectfont{$\bullet$}};
\draw [color=black] (-1.25,9.5) node[] {\fontsize{8 pt}{0pt}\selectfont{$\bullet$}};
\draw [color=black] (-0.25,9.5) node[] {\fontsize{8 pt}{0pt}\selectfont{$\bullet$}};
\draw [color=black] (0.75,9.5) node[] {\fontsize{8 pt}{0pt}\selectfont{$\bullet$}};
\draw [color=black] (1.75,9.5) node[] {\fontsize{8 pt}{0pt}\selectfont{$\bullet$}};
\draw [color=black] (2.75,9.5) node[] {\fontsize{8 pt}{0pt}\selectfont{$\bullet$}};
\draw [color=black] (3.75,9.5) node[] {\fontsize{8 pt}{0pt}\selectfont{$\bullet$}};
\draw [color=black] (4.75,9.5) node[] {\fontsize{8 pt}{0pt}\selectfont{$\bullet$}};
\draw [color=black] (5.75,9.5) node[] {\fontsize{8 pt}{0pt}\selectfont{$\bullet$}};
\draw [color=black] (-1.5,9) node[] {\fontsize{8 pt}{0pt}\selectfont{$\bullet$}};
\draw [color=black] (-1,9) node[] {\fontsize{8 pt}{0pt}\selectfont{$\bullet$}};
\draw [color=black] (-0.5,9) node[] {\fontsize{8 pt}{0pt}\selectfont{$\bullet$}};
\draw [color=black] (0,9) node[] {\fontsize{8 pt}{0pt}\selectfont{$\bullet$}};
\draw [color=black] (0.5,9) node[] {\fontsize{8 pt}{0pt}\selectfont{$\bullet$}};
\draw [color=black] (1,9) node[] {\fontsize{8 pt}{0pt}\selectfont{$\bullet$}};
\draw [color=black] (1.5,9) node[] {\fontsize{8 pt}{0pt}\selectfont{$\bullet$}};
\draw [color=black] (2,9) node[] {\fontsize{8 pt}{0pt}\selectfont{$\bullet$}};
\draw [color=black] (2.5,9) node[] {\fontsize{8 pt}{0pt}\selectfont{$\bullet$}};
\draw [color=black] (3,9) node[] {\fontsize{8 pt}{0pt}\selectfont{$\bullet$}};
\draw [color=black] (3.5,9) node[] {\fontsize{8 pt}{0pt}\selectfont{$\bullet$}};
\draw [color=black] (4,9) node[] {\fontsize{8 pt}{0pt}\selectfont{$\bullet$}};
\draw [color=black] (4.5,9) node[] {\fontsize{8 pt}{0pt}\selectfont{$\bullet$}};
\draw [color=black] (5,9) node[] {\fontsize{8 pt}{0pt}\selectfont{$\bullet$}};
\draw [color=black] (5.5,9) node[] {\fontsize{8 pt}{0pt}\selectfont{$\bullet$}};
\draw [color=black] (6,9) node[] {\fontsize{8 pt}{0pt}\selectfont{$\bullet$}};
\end{tikzpicture}  
  \caption{$\eta$.}
     \end{subfigure}
     \hspace{3pt}
     \begin{subfigure}[b]{0.45\textwidth}
         \centering
         \begin{tikzpicture}[line cap=round,line join=round,>=triangle 45,x=0.6cm,y=0.7cm]
\clip(-2,8.5) rectangle (6.5,11.5);
\draw [line width=0.5pt] (2.25,11)-- (0.25,10.5);
\draw [line width=0.5pt] (2.25,11)-- (4.25,10.5);
\draw [line width=0.5pt] (0.25,10.5)-- (-0.75,10);
\draw [line width=0.5pt] (0.25,10.5)-- (1.25,10);
\draw [line width=0.5pt] (4.25,10.5)-- (3.25,10);
\draw [line width=0.5pt] (4.25,10.5)-- (5.25,10);
\draw [line width=0.5pt] (-0.75,10)-- (-1.25,9.5);
\draw [line width=0.5pt] (-0.75,10)-- (-0.25,9.5);
\draw [line width=0.5pt] (1.25,10)-- (0.75,9.5);
\draw [line width=0.5pt] (1.25,10)-- (1.75,9.5);
\draw [line width=0.5pt] (3.25,10)-- (2.75,9.5);
\draw [line width=0.5pt] (3.25,10)-- (3.75,9.5);
\draw [line width=0.5pt] (5.25,10)-- (4.75,9.5);
\draw [line width=0.5pt] (5.25,10)-- (5.75,9.5);
\draw [line width=0.5pt] (-1.25,9.5)-- (-1.5,9);
\draw [line width=0.5pt] (-1.25,9.5)-- (-1,9);
\draw [line width=0.5pt] (-0.25,9.5)-- (-0.5,9);
\draw [line width=0.5pt] (-0.25,9.5)-- (0,9);
\draw [line width=0.5pt] (0.75,9.5)-- (0.5,9);
\draw [line width=0.5pt] (0.75,9.5)-- (1,9);
\draw [line width=0.5pt] (1.75,9.5)-- (1.5,9);
\draw [line width=0.5pt] (1.75,9.5)-- (2,9);
\draw [line width=0.5pt] (2.75,9.5)-- (2.5,9);
\draw [line width=0.5pt] (2.75,9.5)-- (3,9);
\draw [line width=0.5pt] (3.75,9.5)-- (3.5,9);
\draw [line width=0.5pt] (3.75,9.5)-- (4,9);
\draw [line width=0.5pt] (4.75,9.5)-- (4.5,9);
\draw [line width=0.5pt] (4.75,9.5)-- (5,9);
\draw [line width=0.5pt] (5.75,9.5)-- (5.5,9);
\draw [line width=0.5pt] (5.75,9.5)-- (6,9);
\draw [color=qqqqff] (2.25,11) node[] {\fontsize{8 pt}{0pt}\selectfont{$\bullet$}};
\draw [color=qqqqff] (0.25,10.5) node[] {\fontsize{8 pt}{0pt}\selectfont{$\bullet$}};
\draw [color=qqqqff] (4.25,10.5) node[] {\fontsize{8 pt}{0pt}\selectfont{$\bullet$}};
\draw [color=qqqqff] (-0.75,10) node[] {\fontsize{8 pt}{0pt}\selectfont{$\bullet$}};
\draw [color=qqqqff] (1.25,10) node[] {\fontsize{8 pt}{0pt}\selectfont{$\bullet$}};
\draw [color=qqqqff] (3.25,10) node[] {\fontsize{8 pt}{0pt}\selectfont{$\bullet$}};
\draw [color=qqqqff] (5.25,10) node[] {\fontsize{8 pt}{0pt}\selectfont{$\bullet$}};
\draw [color=qqqqff] (-1.25,9.5) node[] {\fontsize{8 pt}{0pt}\selectfont{$\bullet$}};
\draw [color=qqqqff] (-0.25,9.5) node[] {\fontsize{8 pt}{0pt}\selectfont{$\bullet$}};
\draw [color=qqqqff] (0.75,9.5) node[] {\fontsize{8 pt}{0pt}\selectfont{$\bullet$}};
\draw [color=qqqqff] (1.75,9.5) node[] {\fontsize{8 pt}{0pt}\selectfont{$\bullet$}};
\draw [color=qqqqff] (2.75,9.5) node[] {\fontsize{8 pt}{0pt}\selectfont{$\bullet$}};
\draw [color=qqqqff] (3.75,9.5) node[] {\fontsize{8 pt}{0pt}\selectfont{$\bullet$}};
\draw [color=qqqqff] (4.75,9.5) node[] {\fontsize{8 pt}{0pt}\selectfont{$\bullet$}};
\draw [color=qqqqff] (5.75,9.5) node[] {\fontsize{8 pt}{0pt}\selectfont{$\bullet$}};
\draw [color=qqqqff] (-1.5,9) node[] {\fontsize{8 pt}{0pt}\selectfont{$\bullet$}};
\draw [color=qqqqff] (-1,9) node[] {\fontsize{8 pt}{0pt}\selectfont{$\bullet$}};
\draw [color=qqqqff] (-0.5,9) node[] {\fontsize{8 pt}{0pt}\selectfont{$\bullet$}};
\draw [color=qqqqff] (0,9) node[] {\fontsize{8 pt}{0pt}\selectfont{$\bullet$}};
\draw [color=qqqqff] (0.5,9) node[] {\fontsize{8 pt}{0pt}\selectfont{$\bullet$}};
\draw [color=qqqqff] (1,9) node[] {\fontsize{8 pt}{0pt}\selectfont{$\bullet$}};
\draw [color=qqqqff] (1.5,9) node[] {\fontsize{8 pt}{0pt}\selectfont{$\bullet$}};
\draw [color=qqqqff] (2,9) node[] {\fontsize{8 pt}{0pt}\selectfont{$\bullet$}};
\draw [color=qqqqff] (2.5,9) node[] {\fontsize{8 pt}{0pt}\selectfont{$\bullet$}};
\draw [color=qqqqff] (3,9) node[] {\fontsize{8 pt}{0pt}\selectfont{$\bullet$}};
\draw [color=qqqqff] (3.5,9) node[] {\fontsize{8 pt}{0pt}\selectfont{$\bullet$}};
\draw [color=qqqqff] (4,9) node[] {\fontsize{8 pt}{0pt}\selectfont{$\bullet$}};
\draw [color=qqqqff] (4.5,9) node[] {\fontsize{8 pt}{0pt}\selectfont{$\bullet$}};
\draw [color=qqqqff] (5,9) node[] {\fontsize{8 pt}{0pt}\selectfont{$\bullet$}};
\draw [color=qqqqff] (5.5,9) node[] {\fontsize{8 pt}{0pt}\selectfont{$\bullet$}};
\draw [color=qqqqff] (6,9) node[] {\fontsize{8 pt}{0pt}\selectfont{$\bullet$}};
\end{tikzpicture}
  \caption{Every vertices in $T_2$ have the same follower set $T_2$.}
     \end{subfigure}
      \caption{$T_2$.}
  \label{Fig 1}
\end{figure}

    \item Let $G:=\left[\begin{matrix}
        1&1\\
        1&0
    \end{matrix}\right]$. Since $F_{T_G}(g_1)=T_G$, we have $\mathcal{I}=\{\eta_1, \eta_2\}$, where $\eta_1=T_G$ and $\eta_2=F_{T_G}(g_2)$. Then,
\begin{equation*}
        T^{(\eta_1)}=\{g\in T_G : F_{T_G}(g)=\eta_1\}\mbox{ and }
        T^{(\eta_2)}=\{g\in T_G : F_{T_G}(g)=\eta_2\}.
\end{equation*}

    \begin{figure}[H]
     \centering
     \begin{subfigure}[b]{0.45\textwidth}
         \centering
               \centering
     \begin{subfigure}[b]{0.45\textwidth}
         \centering
         \begin{tikzpicture}[line cap=round,line join=round,x=0.5cm,y=0.7cm]
       \clip(-1.75,5.5) rectangle (4.75,8.5);
\draw [line width=0.5pt] (2.25,8)-- (0.25,7.5);
\draw [line width=0.5pt] (2.25,8)-- (4.25,7.5);
\draw [line width=0.5pt] (0.25,7.5)-- (-0.75,7);
\draw [line width=0.5pt] (0.25,7.5)-- (1.25,7);
\draw [line width=0.5pt] (4.25,7.5)-- (3.25,7);
\draw [line width=0.5pt] (-0.75,7)-- (-1.25,6.5);
\draw [line width=0.5pt] (-0.75,7)-- (-0.25,6.5);
\draw [line width=0.5pt] (1.25,7)-- (0.75,6.5);
\draw [line width=0.5pt] (3.25,7)-- (2.75,6.5);
\draw [line width=0.5pt] (3.25,7)-- (3.75,6.5);
\draw [line width=0.5pt] (-1.25,6.5)-- (-1.5,6);
\draw [line width=0.5pt] (-1.25,6.5)-- (-1,6);
\draw [line width=0.5pt] (-0.25,6.5)-- (-0.5,6);
\draw [line width=0.5pt] (0.75,6.5)-- (0.5,6);
\draw [line width=0.5pt] (0.75,6.5)-- (1,6);
\draw [line width=0.5pt] (2.75,6.5)-- (2.5,6);
\draw [line width=0.5pt] (2.75,6.5)-- (3,6);
\draw [line width=0.5pt] (3.75,6.5)-- (3.5,6);
\draw [color=black] (2.25,8) node[] {\fontsize{8 pt}{0pt}\selectfont{$\bullet$}};
\draw [color=black] (0.25,7.5) node[] {\fontsize{8 pt}{0pt}\selectfont{$\bullet$}};
\draw [color=black] (4.25,7.5) node[] {\fontsize{8 pt}{0pt}\selectfont{$\bullet$}};
\draw [color=black] (-0.75,7) node[] {\fontsize{8 pt}{0pt}\selectfont{$\bullet$}};
\draw [color=black] (1.25,7) node[] {\fontsize{8 pt}{0pt}\selectfont{$\bullet$}};
\draw [color=black] (3.25,7) node[] {\fontsize{8 pt}{0pt}\selectfont{$\bullet$}};
\draw [color=black] (-1.25,6.5) node[] {\fontsize{8 pt}{0pt}\selectfont{$\bullet$}};
\draw [color=black] (-0.25,6.5) node[] {\fontsize{8 pt}{0pt}\selectfont{$\bullet$}};
\draw [color=black] (0.75,6.5) node[] {\fontsize{8 pt}{0pt}\selectfont{$\bullet$}};
\draw [color=black] (2.75,6.5) node[] {\fontsize{8 pt}{0pt}\selectfont{$\bullet$}};
\draw [color=black] (3.75,6.5) node[] {\fontsize{8 pt}{0pt}\selectfont{$\bullet$}};
\draw [color=black] (-1.5,6) node[] {\fontsize{8 pt}{0pt}\selectfont{$\bullet$}};
\draw [color=black] (-1,6) node[] {\fontsize{8 pt}{0pt}\selectfont{$\bullet$}};
\draw [color=black] (-0.5,6) node[] {\fontsize{8 pt}{0pt}\selectfont{$\bullet$}};
\draw [color=black] (0.5,6) node[] {\fontsize{8 pt}{0pt}\selectfont{$\bullet$}};
\draw [color=black] (1,6) node[] {\fontsize{8 pt}{0pt}\selectfont{$\bullet$}};
\draw [color=black] (2.5,6) node[] {\fontsize{8 pt}{0pt}\selectfont{$\bullet$}};
\draw [color=black] (3,6) node[] {\fontsize{8 pt}{0pt}\selectfont{$\bullet$}};
\draw [color=black] (3.5,6) node[] {\fontsize{8 pt}{0pt}\selectfont{$\bullet$}};
\end{tikzpicture}
     \end{subfigure}
     \hfill
     \begin{subfigure}[b]{0.45\textwidth}
         \centering
         \begin{tikzpicture}[line cap=round,line join=round,x=0.5cm,y=0.7cm]
       \clip(-2,5.5) rectangle (2.5,8.5);
\draw [line width=0.5pt] (2.25,8)-- (0.25,7.5);
\draw [line width=0.5pt] (0.25,7.5)-- (-0.75,7);
\draw [line width=0.5pt] (0.25,7.5)-- (1.25,7);
\draw [line width=0.5pt] (-0.75,7)-- (-1.25,6.5);
\draw [line width=0.5pt] (-0.75,7)-- (-0.25,6.5);
\draw [line width=0.5pt] (1.25,7)-- (0.75,6.5);
\draw [line width=0.5pt] (-1.25,6.5)-- (-1.5,6);
\draw [line width=0.5pt] (-1.25,6.5)-- (-1,6);
\draw [line width=0.5pt] (-0.25,6.5)-- (-0.5,6);
\draw [line width=0.5pt] (0.75,6.5)-- (0.5,6);
\draw [line width=0.5pt] (0.75,6.5)-- (1,6);
\draw [color=black] (2.25,8) node[] {\fontsize{8 pt}{0pt}\selectfont{$\bullet$}};
\draw [color=black] (0.25,7.5) node[] {\fontsize{8 pt}{0pt}\selectfont{$\bullet$}};
\draw [color=black] (-0.75,7) node[] {\fontsize{8 pt}{0pt}\selectfont{$\bullet$}};
\draw [color=black] (1.25,7) node[] {\fontsize{8 pt}{0pt}\selectfont{$\bullet$}};
\draw [color=black] (-1.25,6.5) node[] {\fontsize{8 pt}{0pt}\selectfont{$\bullet$}};
\draw [color=black] (-0.25,6.5) node[] {\fontsize{8 pt}{0pt}\selectfont{$\bullet$}};
\draw [color=black] (0.75,6.5) node[] {\fontsize{8 pt}{0pt}\selectfont{$\bullet$}};
\draw [color=black] (-1.5,6) node[] {\fontsize{8 pt}{0pt}\selectfont{$\bullet$}};
\draw [color=black] (-1,6) node[] {\fontsize{8 pt}{0pt}\selectfont{$\bullet$}};
\draw [color=black] (-0.5,6) node[] {\fontsize{8 pt}{0pt}\selectfont{$\bullet$}};
\draw [color=black] (0.5,6) node[] {\fontsize{8 pt}{0pt}\selectfont{$\bullet$}};
\draw [color=black] (1,6) node[] {\fontsize{8 pt}{0pt}\selectfont{$\bullet$}};
\end{tikzpicture}
     \end{subfigure}
  \caption{$\eta_1$ and $\eta_2$.}
     \end{subfigure}
     %\hfill
     \begin{subfigure}[b]{0.45\textwidth}
         \centering
         \begin{tikzpicture}[line cap=round,line join=round,>=triangle 45,x=0.6cm,y=0.7cm]
       \clip(-1.75,5.5) rectangle (4.75,8.5);
\draw [line width=0.5pt] (2.25,8)-- (0.25,7.5);
\draw [line width=0.5pt] (2.25,8)-- (4.25,7.5);
\draw [line width=0.5pt] (0.25,7.5)-- (-0.75,7);
\draw [line width=0.5pt] (0.25,7.5)-- (1.25,7);
\draw [line width=0.5pt] (4.25,7.5)-- (3.25,7);
\draw [line width=0.5pt] (-0.75,7)-- (-1.25,6.5);
\draw [line width=0.5pt] (-0.75,7)-- (-0.25,6.5);
\draw [line width=0.5pt] (1.25,7)-- (0.75,6.5);
\draw [line width=0.5pt] (3.25,7)-- (2.75,6.5);
\draw [line width=0.5pt] (3.25,7)-- (3.75,6.5);
\draw [line width=0.5pt] (-1.25,6.5)-- (-1.5,6);
\draw [line width=0.5pt] (-1.25,6.5)-- (-1,6);
\draw [line width=0.5pt] (-0.25,6.5)-- (-0.5,6);
\draw [line width=0.5pt] (0.75,6.5)-- (0.5,6);
\draw [line width=0.5pt] (0.75,6.5)-- (1,6);
\draw [line width=0.5pt] (2.75,6.5)-- (2.5,6);
\draw [line width=0.5pt] (2.75,6.5)-- (3,6);
\draw [line width=0.5pt] (3.75,6.5)-- (3.5,6);
\draw [color=qqqqff] (2.25,8) node[] {\fontsize{8 pt}{0pt}\selectfont{$\bullet$}};
\draw [color=qqqqff] (0.25,7.5) node[] {\fontsize{8 pt}{0pt}\selectfont{$\bullet$}};
\draw [color=ffqqqq] (4.25,7.5) node[] {\fontsize{4 pt}{0pt}\selectfont{$\blacksquare$}};
\draw [color=qqqqff] (-0.75,7) node[] {\fontsize{8 pt}{0pt}\selectfont{$\bullet$}};
\draw [color=ffqqqq] (1.25,7) node[] {\fontsize{4 pt}{0pt}\selectfont{$\blacksquare$}};
\draw [color=qqqqff] (3.25,7) node[] {\fontsize{8 pt}{0pt}\selectfont{$\bullet$}};
\draw [color=qqqqff] (-1.25,6.5) node[] {\fontsize{8 pt}{0pt}\selectfont{$\bullet$}};
\draw [color=ffqqqq] (-0.25,6.5) node[] {\fontsize{4 pt}{0pt}\selectfont{$\blacksquare$}};
\draw [color=qqqqff] (0.75,6.5) node[] {\fontsize{8 pt}{0pt}\selectfont{$\bullet$}};
\draw [color=qqqqff] (2.75,6.5) node[] {\fontsize{8 pt}{0pt}\selectfont{$\bullet$}};
\draw [color=ffqqqq] (3.75,6.5) node[] {\fontsize{4 pt}{0pt}\selectfont{$\blacksquare$}};
\draw [color=qqqqff] (-1.5,6) node[] {\fontsize{8 pt}{0pt}\selectfont{$\bullet$}};
\draw [color=qqqqff] (-1,6) node[] {\fontsize{8 pt}{0pt}\selectfont{$\bullet$}};
\draw [color=qqqqff] (-0.5,6) node[] {\fontsize{8 pt}{0pt}\selectfont{$\bullet$}};
\draw [color=qqqqff] (0.5,6) node[] {\fontsize{8 pt}{0pt}\selectfont{$\bullet$}};
\draw [color=qqqqff] (1,6) node[] {\fontsize{8 pt}{0pt}\selectfont{$\bullet$}};
\draw [color=qqqqff] (2.5,6) node[] {\fontsize{8 pt}{0pt}\selectfont{$\bullet$}};
\draw [color=qqqqff] (3,6) node[] {\fontsize{8 pt}{0pt}\selectfont{$\bullet$}};
\draw [color=qqqqff] (3.5,6) node[] {\fontsize{8 pt}{0pt}\selectfont{$\bullet$}};
\end{tikzpicture}
  \caption{The vertices denoted as {\color{blue}\scalebox{1.4}{$\bullet$}} (resp. {\color{red}\scalebox{0.9}{$\blacksquare$}}) are belongs to $T^{(\eta_1)}$  (resp. $T^{(\eta_2)}$)}
     \end{subfigure}
        \caption{$T_G$}
        \label{Fig 2}
\end{figure}

  \item Let $U:=\left[\begin{matrix}
        1&1\\
        0&1
    \end{matrix}\right]$. Since $F_{T_U}(g_1)=T_U$, we have $\mathcal{I}=\{\eta_1, \eta_2\}$, where $\eta_1=T_U$ and $\eta_2=F_{T_U}(g_2)$. Then, \[T^{(\eta_1)}=\{g\in T_U : F_{T_U}(g)=\eta_1\}\mbox{ and }    T^{(\eta_2)}=\{g\in T_U : F_{T_U}(g)=\eta_2\}.\]

 \begin{figure}[H]
     \centering
     \begin{subfigure}[b]{0.45\textwidth}
         \centering
               \centering
     \begin{subfigure}[b]{0.45\textwidth}
         \centering
        \begin{tikzpicture}[line cap=round,line join=round,>=triangle 45,x=0.4cm,y=0.7cm]
            \clip(-1.75,2.5) rectangle (6.25,5.5);
\draw [line width=0.5pt] (2.25,5)-- (0.25,4.5);
\draw [line width=0.5pt] (2.25,5)-- (4.25,4.5);
\draw [line width=0.5pt] (0.25,4.5)-- (-0.75,4);
\draw [line width=0.5pt] (0.25,4.5)-- (1.25,4);
\draw [line width=0.5pt] (4.25,4.5)-- (5.25,4);
\draw [line width=0.5pt] (-0.75,4)-- (-1.25,3.5);
\draw [line width=0.5pt] (-0.75,4)-- (-0.25,3.5);
\draw [line width=0.5pt] (1.25,4)-- (1.75,3.5);
\draw [line width=0.5pt] (5.25,4)-- (5.75,3.5);
\draw [line width=0.5pt] (-1.25,3.5)-- (-1.5,3);
\draw [line width=0.5pt] (-1.25,3.5)-- (-1,3);
\draw [line width=0.5pt] (-0.25,3.5)-- (0,3);
\draw [line width=0.5pt] (1.75,3.5)-- (2,3);
\draw [line width=0.5pt] (5.75,3.5)-- (6,3);
\draw [color=black] (2.25,5) node[] {\fontsize{8 pt}{0pt}\selectfont{$\bullet$}};
\draw [color=black] (0.25,4.5) node[] {\fontsize{8 pt}{0pt}\selectfont{$\bullet$}};
\draw [color=black] (4.25,4.5) node[] {\fontsize{8 pt}{0pt}\selectfont{$\bullet$}};
\draw [color=black] (-0.75,4) node[] {\fontsize{8 pt}{0pt}\selectfont{$\bullet$}};
\draw [color=black] (1.25,4) node[] {\fontsize{8 pt}{0pt}\selectfont{$\bullet$}};
\draw [color=black] (5.25,4) node[] {\fontsize{8 pt}{0pt}\selectfont{$\bullet$}};
\draw [color=black] (-1.25,3.5) node[] {\fontsize{8 pt}{0pt}\selectfont{$\bullet$}};
\draw [color=black] (-0.25,3.5) node[] {\fontsize{8 pt}{0pt}\selectfont{$\bullet$}};
\draw [color=black] (1.75,3.5) node[] {\fontsize{8 pt}{0pt}\selectfont{$\bullet$}};
\draw [color=black] (5.75,3.5) node[] {\fontsize{8 pt}{0pt}\selectfont{$\bullet$}};
\draw [color=black] (-1.5,3) node[] {\fontsize{8 pt}{0pt}\selectfont{$\bullet$}};
\draw [color=black] (-1,3) node[] {\fontsize{8 pt}{0pt}\selectfont{$\bullet$}};
\draw [color=black] (0,3) node[] {\fontsize{8 pt}{0pt}\selectfont{$\bullet$}};
\draw [color=black] (2,3) node[] {\fontsize{8 pt}{0pt}\selectfont{$\bullet$}};
\draw [color=black] (6,3) node[] {\fontsize{8 pt}{0pt}\selectfont{$\bullet$}};
         \end{tikzpicture}
     \end{subfigure}
     \hfill
     \begin{subfigure}[b]{0.45\textwidth}
         \centering
         \begin{tikzpicture}[line cap=round,line join=round,>=triangle 45,x=0.4cm,y=0.7cm]
            \clip(2,2.5) rectangle (6.25,5.5);
\draw [line width=0.5pt] (2.25,5)-- (4.25,4.5);
\draw [line width=0.5pt] (4.25,4.5)-- (5.25,4);
\draw [line width=0.5pt] (5.25,4)-- (5.75,3.5);
\draw [line width=0.5pt] (5.75,3.5)-- (6,3);
\draw [color=black] (2.25,5) node[] {\fontsize{8 pt}{0pt}\selectfont{$\bullet$}};
\draw [color=black] (4.25,4.5) node[] {\fontsize{8 pt}{0pt}\selectfont{$\bullet$}};
\draw [color=black] (5.25,4) node[] {\fontsize{8 pt}{0pt}\selectfont{$\bullet$}};
\draw [color=black] (5.75,3.5) node[] {\fontsize{8 pt}{0pt}\selectfont{$\bullet$}};
\draw [color=black] (6,3) node[] {\fontsize{8 pt}{0pt}\selectfont{$\bullet$}};
         \end{tikzpicture}
     \end{subfigure}
  \caption{$\eta_1$ and $\eta_2$.}
     \end{subfigure}
     %\hfill
     \begin{subfigure}[b]{0.45\textwidth}
         \centering
         \begin{tikzpicture}[line cap=round,line join=round,>=triangle 45,x=0.6cm,y=0.7cm]
     \clip(-1.75,2.5) rectangle (6.25,5.5);
\draw [line width=0.5pt] (2.25,5)-- (0.25,4.5);
\draw [line width=0.5pt] (2.25,5)-- (4.25,4.5);
\draw [line width=0.5pt] (0.25,4.5)-- (-0.75,4);
\draw [line width=0.5pt] (0.25,4.5)-- (1.25,4);
\draw [line width=0.5pt] (4.25,4.5)-- (5.25,4);
\draw [line width=0.5pt] (-0.75,4)-- (-1.25,3.5);
\draw [line width=0.5pt] (-0.75,4)-- (-0.25,3.5);
\draw [line width=0.5pt] (1.25,4)-- (1.75,3.5);
\draw [line width=0.5pt] (5.25,4)-- (5.75,3.5);
\draw [line width=0.5pt] (-1.25,3.5)-- (-1.5,3);
\draw [line width=0.5pt] (-1.25,3.5)-- (-1,3);
\draw [line width=0.5pt] (-0.25,3.5)-- (0,3);
\draw [line width=0.5pt] (1.75,3.5)-- (2,3);
\draw [line width=0.5pt] (5.75,3.5)-- (6,3);
\draw [color=qqqqff] (2.25,5) node[] {\fontsize{8 pt}{0pt}\selectfont{$\bullet$}};
\draw [color=qqqqff] (0.25,4.5) node[] {\fontsize{8 pt}{0pt}\selectfont{$\bullet$}};
\draw [color=ffqqqq] (4.25,4.5) node[] {\fontsize{4 pt}{0pt}\selectfont{$\blacksquare$}};
\draw [color=qqqqff] (-0.75,4) node[] {\fontsize{8 pt}{0pt}\selectfont{$\bullet$}};
\draw [color=ffqqqq] (1.25,4) node[] {\fontsize{4 pt}{0pt}\selectfont{$\blacksquare$}};
\draw [color=ffqqqq] (5.25,4) node[] {\fontsize{4 pt}{0pt}\selectfont{$\blacksquare$}};
\draw [color=qqqqff] (-1.25,3.5) node[] {\fontsize{8 pt}{0pt}\selectfont{$\bullet$}};
\draw [color=ffqqqq] (-0.25,3.5) node[] {\fontsize{4 pt}{0pt}\selectfont{$\blacksquare$}};
\draw [color=ffqqqq] (1.75,3.5) node[] {\fontsize{4 pt}{0pt}\selectfont{$\blacksquare$}};
\draw [color=ffqqqq] (5.75,3.5) node[] {\fontsize{4 pt}{0pt}\selectfont{$\blacksquare$}};
\draw [color=qqqqff] (-1.5,3) node[] {\fontsize{8 pt}{0pt}\selectfont{$\bullet$}};
\draw [color=ffqqqq] (-1,3) node[] {\fontsize{4 pt}{0pt}\selectfont{$\blacksquare$}};
\draw [color=ffqqqq] (0,3) node[] {\fontsize{4 pt}{0pt}\selectfont{$\blacksquare$}};
\draw [color=ffqqqq] (2,3) node[] {\fontsize{4 pt}{0pt}\selectfont{$\blacksquare$}};
\draw [color=ffqqqq] (6,3) node[] {\fontsize{4 pt}{0pt}\selectfont{$\blacksquare$}};
\end{tikzpicture}
  \caption{The vertices denoted as {\color{blue}\scalebox{1.4}{$\bullet$}} (resp. {\color{red}\scalebox{0.9}{$\blacksquare$}}) are belongs to $T^{(\eta_1)}$  (resp. $T^{(\eta_2)}$)}
     \end{subfigure}
        \caption{$T_U$}
        \label{Fig 3}
\end{figure}
\end{enumerate}
\end{example}

Let $\mathcal{A}$ be a set of finite symbols and let $\mathcal{F}\subseteq \bigcup\limits_{m=0}^{\infty }\bigcup\limits_{\eta
\in \mathcal{I}}\mathcal{A}^{\Delta _{m}^{(\eta )}}$. The \emph{tree-shift }$\mathcal{T}_{\mathcal{F}}$ (or $\mathcal{T}$) of $T$ associated with the \emph{forbidden set} $\mathcal{F}$ is defined by 
\begin{equation}
\left\{ t=(t_g)_{g\in T}\in \mathcal{A}^{T}:\sigma
_{g}(t)|_{\Delta _{m}^{(\eta )}}\notin \mathcal{F}\text{, }\forall \eta \in 
\mathcal{I}\text{, }\forall g\in T^{(\eta )}\text{, }\forall m\in \mathbb{N}\cup \{0\}\right\} \text{,}  \label{10}
\end{equation}
where $\left( \sigma _{g}(t)\right) _{g^{\prime }}=t_{gg^{\prime }}$ ($\forall
g^{\prime }\in \eta ,\forall g\in T^{(\eta )},\forall \eta\in\mathcal{I}$), and $t|_{F}=(t_{g})_{g\in F}$ is the canonical projection of $t\in \mathcal{A}^{T}$ into $\mathcal{A}^{F}$ ($\forall F\subseteq T$). When the set $\mathcal{F}$ is finite, we call $\mathcal{T}$ a \emph{tree-shift of finite type }(TSFT for short). In view of (\ref{10}), we define the \emph{projected tree-shift} of $\mathcal{T}$ by
\[
\mathcal{T}^{\mathcal{I}}=\bigsqcup\limits_{\eta \in \mathcal{I}}\mathcal{T}^{(\eta )}\text{,}
\]
where $\sqcup$ denotes the disjoint union, and
\[
\mathcal{T}^{(\eta )}=\bigcup\limits_{g\in T^{(\eta )}}\{\sigma _{g}(t):t\in\mathcal{T}\}.
\]

Note that if $T=T_{d}$, then $\left\vert \mathcal{I}\right\vert =1$, $\Delta_{m}^{(\eta)}=\{g\in T_d: 0\leq |g-\epsilon|\leq m \}~(\forall \eta\in\mathcal{I})$, and $\mathcal{T}^{\mathcal{I}}=\mathcal{T}$. Let $\Delta _{m}:=\Delta_m^{(T_d)}$ and let $\mathcal{F}\subseteq \bigcup\limits_{m=0}^{\infty }\mathcal{A}^{\Delta _{m}}$, and the tree-shift $\mathcal{T}$ in (\ref{10}) can be written as 
\begin{equation}
\left\{ t\in \mathcal{A}^{T_d}:\sigma_g(t)|_{\Delta_{m}}\notin \mathcal{F}\text{, }\forall g\in T_d\text{, }\forall m\in \mathbb{N}\cup \{0\}\right\} \text{.}  \label{9}
\end{equation}
The definition of (\ref{9}) is the formal definition of the tree-shift on $T_{d}$ (cf.\cite{aubrun2012tree, ban2017mixing, ban2017tree,PS-2017complexity,petersen2020entropy}). The
research topics on the dynamics of $\mathcal{T}$ have been receiving a lot of attention lately because they exhibit phenomena that are very different from shift spaces defined on $\mathbb{Z}^{d}$. Let $\mathcal{T}\subseteq \mathcal{A}^{T}$ be a tree-shift on the tree $T$, and the \emph{distance} $d(s,t)$ of $s$, $t\in \mathcal{T}^{(\eta )}$ for $\eta \in \mathcal{I}$ is defined by 
\begin{equation}
    d(s,t)=\inf \left\{2^{-m}:s|_{\Delta _{m}^{(\eta )}}=t|_{\Delta _{m}^{(\eta )}}\right\}\text{.} \label{1}
\end{equation}

\subsection{The $\omega$-limit sets}
In this subsection, we study the $\omega$-limit set of a shift space defined on $G$. Let $G$ be a group or monoid, $F$ be a family of subsets of $G$, $\sigma_g(x):G\times X\to X$ ($g\in G$ and $x\in X$) be a $G$ action on a compact metric space $X$, and $x\in X$. The $\omega$\emph{-limit set of $x$ for the family }$F$ \cite{souza2012limit} is defined by 
\begin{equation}
\omega (x)=\bigcap\limits_{A\in F}\overline{\{\sigma_{g}(x):g\in A\}}\text{.}  \label{2}
\end{equation}

Binder and Meddaugh \cite{binder2019limitPH,binder2019limit} introduced various types of $\omega $-limit sets of a tree-shift $\mathcal{T}$ on $T_{d}$ in Definition \ref{def1.2}, based on (\ref{2}) with $G=T_d$. We say $p$ is a \emph{ray} in $T$ if $p$ is an infinite non-self-intersection path emanating from $\epsilon$, that is,
\[
p=( p_{0}, p_{0}p_{1}, p_{0}p_{1}p_{2}, ...) = ( \epsilon, g_{i_{1}}, g_{i_{1}}g_{i_{2}}, ...).
\]
We denote the collection of all rays in $T$ by $\partial T$.

\begin{definition}\label{def1.2}
Let $\mathcal{T}$ be a tree-shift on $T_{d}$, $t\in \mathcal{T}$, and $p\in \partial T_d$. 
\begin{enumerate}
\item The $\omega$\emph{-limit set of }$t$ is defined as 
\begin{equation}\label{3}
\omega^{d}(t)=\bigcap\limits_{n\in \mathbb{N}}\overline{\{\sigma_{g}(t): g\in T_{d}\text{, }\left\vert g-\epsilon \right\vert >n\}}\text{.}
\end{equation}

\item The $\omega$\emph{-limit set of }$t$\emph{\ along }$p$ is defined as 
\begin{equation}
\omega _{p}^{d}(t)=\bigcap\limits_{n\in \mathbb{N}}\overline{\{\sigma_{g_{i_{1}}g_{i_{2}}\cdots g_{i_{k}}}(t):k\geq n\}}\text{.}  \label{4}
\end{equation}

\item The $\omega $\emph{-limit set of }$t$\emph{\ along followers of }$p$ is defined as
\begin{equation}
\omega _{F_{p}}^{d}(t)=\bigcap\limits_{n\in \mathbb{N}}\overline{\{\sigma
_{g}(t):g=g_{i_{1}}g_{i_{2}}\cdots g_{i_{n}}g^{\prime }\}}\text{.}  \label{7}
\end{equation}

\item The $\omega$\emph{-limit set $\omega _{CPS}^{d}(t)$ of }$t$\emph{\ in the sense of CPS}\footnote{A \emph{prefix set} $P$ is a subset of $T$ such that no element in $P$ is a prefix of another. A finite prefix $P$ is a \emph{complete prefix set} (CPS) if $\forall g\in T$ with $\left\vert g-\epsilon\right\vert \geq \max \{\left\vert g-\epsilon\right\vert :g\in P\}$ has a prefix in $P$.} is defined as 
\begin{equation}
\{s\in \mathcal{T}:\forall n\text{ }\exists~
C_{n}\subseteq \mathcal{C}\text{ with }\underline{m}_{C_{n}}\geq n\ni
\forall g\in C_{n}\text{ }d(\sigma _{g}(t),s)<n^{-1}\}\text{,}  \label{8}
\end{equation}
where $\mathcal{C}$ is the collection of all CPS of $T$, and $\underline{m}_{P}:=\min \{\left\vert g-\epsilon\right\vert :g\in P\}$.
\item A set $S\subseteq \mathcal{A}^{T_{d}}$ is called \emph{invariant }if $\forall s\in S$, $\sigma _{g_{i}}(s)\in S$ for all $1\leq i\leq d$.
\end{enumerate}
\end{definition}

For $\mathcal{T}$ is a tree-shift on $T_{M}$ and $t\in\mathcal{T}$. The definition of an $\omega$-limit set $\omega^{M}(t)$ (resp. $\omega_{p}^{M}(t)$ and $\omega_{F_{p}}^{M}(t)$) is similar to that of (\ref{3}) (resp. (\ref{4}) and (\ref{7})). The only modification is to
replace the role of $T_{d}$ in each definition of $\omega$-limit sets with $T_{M}$. However, the definition of an $\omega$-limit set of $t$ in the sense of CPS needs to be modified as follows.

\begin{definition}
Let $\mathcal{I}$ be the corresponding family of follower sets of a Markov-Cayley tree  $T_{M}$. For any $\xi_{1}, \ldots ,\xi_{\ell}\in \mathcal{I}$ with $1\leq \ell \leq | \mathcal{I}|$, a subset $C$ of $T_M$ is called a $[\xi_{1}: \xi_{\ell}]$\emph{-CPS} if $C$ is a CPS in $T_{M}$ and $\{ F_{T_M}(g):g\in C\}=\{ \xi_1, ..., \xi_\ell\}$. Let $\mathcal{C}^{[\xi_{1}:\xi_{\ell}]}$ be the collection of all $[\xi_{1}:\xi_{\ell}]$-CPS in $T_{M}$. The $\omega $\emph{-limit set of }$t$\emph{ in the sense of CPS} is defined as 
\[
\omega _{CPS}^{M}(t)=\bigcup\limits_{\ell=1}^{\left\vert \mathcal{I}\right\vert}\bigcup\limits_{\substack{\xi_{1},\ldots ,\xi_{\ell}\in \mathcal{I}\\ \forall i\neq j,~\xi_{i}\neq \xi_{j}}}\omega_{CPS}^{M;[\xi_{1}:\xi_{\ell}]}(t)\text{,} 
\]
where $\omega_{CPS}^{M;[\xi _{1}:\xi _{\ell}]}(t)$ is the set of vectors $(t_{1}, \ldots ,t_{\ell})\in \prod_{i=1}^{\ell}\mathcal{A}^{\xi _{i}}$ satisfying that for all $n>0$ $\exists$ $C_{n}\in\mathcal{C}^{[\xi_{1}:\xi_{\ell}]}$ with $\underline{m}_{C_{n}}\geq n$ such that $\forall g\in C_{n}$ if $F_{T_{M}}(g)=\xi_{i}$, then $d(\sigma_{g}(t),t_{i})<n^{-1}$. 
\end{definition}
Note that if $M$ is a $d\times d$ full one matrix, then $\mathcal{I}_{T_M}=\{T_d\}$. Thus, $\omega_{CPS}^M(t)=\omega_{CPS}^d(t)$ for all $t\in \mathcal{T}\subseteq \mathcal{A}^{T_M}$. A set $S\subseteq \bigcup\limits_{\eta \in \mathcal{I}}\mathcal{A}^{\eta }$ is 
\emph{invariant }if for all $s\in S$, $\sigma_{g_{i}}(s)\in S$ ($\forall \sigma_{g_{i}}(s)\neq \emptyset $). In \cite[Theorem 4.20]{binder2019limitPH}, the author establishes the basic properties of those $\omega$-limit sets and proved that 
\begin{equation}
\omega_{CPS}^{d}(t)\subseteq \omega_{p}^{d}(t)\subseteq \omega_{F_{p}}^{d}(t)\subseteq \omega^{d}(t),~\forall p\in \partial T_{d},~\forall t\in \mathcal{T}\text{.}  \label{5}
\end{equation}

First, we prove some basic properties of these $\omega$\emph{-limit sets} and remark that the first inclusion of (\ref{5}) is not generally true only when $T_{M}=T_{d}$ (Theorem \ref{Thm: 1} (4) is the same result as \cite[Theorem 4.20]{binder2019limitPH}). For the relation of $\omega_{CPS}^{M}(t)$ with other $\omega$-limit sets, we have the following theorem. Let $A$ be a collection of vectors in $\mathbb{R}^{i}$ for some $i\in \mathbb{N}$. If all coordinates of any vectors of $A$ are in the set $B$, then we write $A\trianglelefteq B$. Otherwise, $A\ntrianglelefteq B$.

\begin{theorem}\label{Thm: 1}
Let $T_M$ be a Markov-Cayley tree, $\mathcal{T}\subseteq \mathcal{A}^{T_{M}}$ be a tree-shift, $p$ be a ray in $\partial T_{M}$ and $t\in \mathcal{T}$. The following assertions hold true.
\begin{enumerate}
\item $\omega ^{M}(t)$ and $\omega_{F_{p}}^{M}(t)$are invariant.

\item $\omega _{p}^{M}(t)\subseteq \omega _{F_{p}}^{M}(t)\subseteq \omega^{M}(t)\subseteq\mathcal{T}^{\mathcal{I}}$.

\item $\omega_{CPS}^{M}(t)\trianglelefteq \omega^{M}(t)$

\item If $M$ is a full matrix, then
\begin{equation}
\omega_{CPS}^{M}(t)\trianglelefteq\omega_{p}^{M}(t)\subseteq \omega_{F_{p}}^{M}(t)\subseteq \omega^{M}(t)\text{.}  \label{6}
\end{equation}
\end{enumerate}
\end{theorem}
 
Example \ref{ex2} illustrates Theorem \ref{Thm: 1} and also provides an example that shows there exists a Markov-Cayley tree $T_U$, a tree-shift $\mathcal{T}\subseteq \mathcal{A}^{T_{U}}$, two rays $p, q\in \partial T_{M}$, and a $t\in \mathcal{T}$ such that
\begin{enumerate}
    \item $\omega_q^U(t)$ is not invariant.
    \item  $\omega _{p}^{U}(t)\nsubseteq \mathcal{T}$ and $\omega_{CPS}^{U}(t) \ntrianglelefteq \mathcal{T}$.    
    \item $\omega_{CPS}^{U}(t) \ntrianglelefteq\omega _{F_{p}}^{U}(t)$.
\end{enumerate}

\begin{example}
Let $T_U$ be a Markov tree (see Figure \ref{Fig 3}). Let $\mathcal{T}=\{0,1\}^{T_U}$, $t=0^{T_U}\in\mathcal{T}$, $p=(\epsilon, g_2,g_2^2, ...)$ and $q=(\epsilon, g_1, g_1^2, ...)$ be two rays in $T_U$. Then, 

\begin{equation*}
    \begin{aligned}
        \omega^U(t)&=\{0^{T_U},0^{F_{T_U}(g_2)}\},\\
        \omega_p^U(t)&=\{ 0^{F_{T_U}(g_2)}\},\\
        \omega_q^U(t)&=\{0^{T_U}\},\\
        \omega_{F_p}^U(t)&=\{0^{F_{T_U}(g_2)}\},\\
        \omega_{CPS}^U(t)&=\{(0^{T_U}, 0^{F_{T_U}(g_2)})\}.       \end{aligned}
\end{equation*}
Since $\omega_q^U(t)=\{0^{T_U}\}$ and $\sigma_{g_2}(0^{T_U})=0^{F_{T_U}(g_2)}\notin \omega_q^U(t)$, we have that $\omega_q^U(t)$ is not invariant.
     \label{ex2} 
\end{example}

Let $\mathcal{T}$ be a tree-shift on $T_{d}$, it is proved that the set of $\{\omega _{p}^{d}(t):t\in \mathcal{T}\}$ possesses a fine structure, say internally chain transitive (ICT).

\begin{definition}[ICT property for shifts on $T_{d}$]
Let $\mathcal{T}$ be a tree-shift on $T_{d}$.

\begin{enumerate}
\item Given $\varepsilon >0$ and an $g=g_{i_{1}}g_{i_{2}}\cdots g_{i_{n}}\in T_{d}$, an $\varepsilon$\emph{-chain indexed by }$g$ is a sequence $\{t_{1},\ldots ,t_{n+1}\}$ of $\mathcal{T}$ such that $d(\sigma
_{g_{i_{j}}}(t_{j}),t_{j+1})<\varepsilon,~\forall 1\leq j\leq n$.

\item A closed subset $Y$ of $\mathcal{T}$ is \emph{internally chain transitive }($Y\in$ ICT) if for every $y$, $y^{\prime }\in Y$ and $\varepsilon >0$, there exist a $g\in T_{d}$ and an $\varepsilon $-chain $\{t_1, ..., t_{n+1}\}$ indexed by $g$ with $t_{1}=y$ and $t_{n+1}=y^{\prime}$.
\end{enumerate}
\end{definition}

In \cite[Theorem 4.23]{binder2019limitPH}, the author shows that the set ICT is closed, and in \cite[Theorem 4.25]{binder2019limitPH}, the author proves that the collection of $\omega_{p}$-limit sets in $\mathcal{T}$ belong to the set of ICT and equals to ICT when $\mathcal{T}$ is a TSFT \cite[Corollary 4.34]{binder2019limitPH}. Our goal is to extend those previous results to the class of $\mathcal{T}^{\prime}$s, which are defined on the Markov-Cayley
tree $T_{M}$. In this situation, the definition of ICT is no longer valid and we demonstrate the suitable definition below.

\begin{definition}[PICT for shifts on $T_{M}$]
Let $M$ be an irreducible matrix and $T_{M}$ is the corresponding Markov-Cayley tree.

\begin{enumerate}
\item For $\varepsilon >0$ and $g=g_{i_{1}}g_{i_{2}}\cdots g_{i_{n}}\in T_{M} $, a $\varepsilon$\emph{-projected chain indexed by }$g$ on $\mathcal{T}^{\mathcal{I}}$ is a sequence $\{t_{1},\ldots ,t_{n+1}\}$ of $\mathcal{T}^{\mathcal{I}}$ such that $d(\sigma_{g_{i_{j}}}(t_{j}),t_{j+1})<\varepsilon, \forall 1\leq j\leq n$.

\item A closed subset $\bigcup\limits_{\eta \in \mathcal{I}}Y^{(\eta )}$ of $\bigcup\limits_{\eta \in \mathcal{I}}\mathcal{T}^{(\eta)}=\mathcal{T}^{\mathcal{I}}$ is \emph{projected internally chain transitive }(write $\bigcup\limits_{\eta \in \mathcal{I}}Y^{(\eta )}\in $ PICT) if for every $y$, $y^{\prime}\in \bigcup\limits_{\eta \in \mathcal{I}}Y^{(\eta )}$ and $\varepsilon >0$, there exist a $g\in T_{M}$ and an $\varepsilon$-projected chain $\{t_1, ..., t_{n+1}\}$ of $\bigcup\limits_{\eta \in \mathcal{I}}Y^{(\eta )}$ indexed by $g$ with $t_{1}=y$ and $t_{n+1}=y^{\prime }$.
\end{enumerate}
\end{definition}
\begin{example}
     For $m,n\geq 1$, let $g=g_1^mg_2^n\in T_U$. Then, we have $F_{T_U}(g_1^i)=T_U\neq F_{T_U}(g_2)=F_{T_U}(g_1^m g_2^j)$ for all $1\leq i \leq m$ and $1\leq j\leq n$ (see Figure \ref{Fig 3}). This provides a rationale for defining PICT as replacing $\mathcal{T}$ in ICT with $\mathcal{T}^{\mathcal{I}}$. 
\end{example}
We have the following results.

\begin{theorem}\label{Thm: 5}
Let $\mathcal{T}$ be a tree-shift defined on a Markov-Cayley tree $T_{M}$. The following assertions hold true.

\begin{enumerate}
\item PICT is closed.

\item $\mathcal{W}_{p}^{M}:=\{\omega _{p}^{M}(t):t\in \mathcal{T}\}\subseteq PICT$.

\item If $\mathcal{T}$ is a TSFT, then $\mathcal{W}_{p}^{M}=PICT$.
\end{enumerate}
\end{theorem}

\subsection{Shadowing and asymptotically shadowing properties}

The (asymptotically) shadowing property for tree-shifts on $T_{d}$ is introduced.

\begin{definition}
Let $\mathcal{T}$ be a tree-shift defined on $T_{d}$.

\begin{enumerate}
\item An \emph{orbit} of $\mathcal{T}$ is a function $\mathcal{O}:T_{d}\rightarrow \mathcal{T}$ such that $\mathcal{O}(g)\in \mathcal{T}$ for all $g\in T_{d}$.

\item For $\delta >0$, a \emph{$\delta$-pseudo orbit} is a function $\mathcal{O}:T_{d}\rightarrow \mathcal{T}$ such that for $g\in T_{d}$ and $g_{i}\in\Sigma $ we have $d(\sigma_{g_{i}}(\mathcal{O}(g)),\mathcal{O}(gg_{i}))<\delta $.

\item A $\delta $-pseudo orbit $\mathcal{O}$ is $\varepsilon $\emph{-shadowed} by a point $t\in \mathcal{T}$ for some $\varepsilon >0$ if $d(\sigma _{g}(t),\mathcal{O}(g))<\varepsilon $ for all $g\in T_{d}$.

\item A tree-shift $\mathcal{T}$ has the \emph{shadowing property }if for all $\varepsilon >0$ there is a $\delta >0$ such that for any $\delta $-pseudo orbit $\mathcal{O}$ of $\mathcal{T}$ is $\varepsilon $-shadowed by a point $t\in\mathcal{T}$.
\end{enumerate}
\end{definition}

For $\mathcal{T}$ is a tree-shift on $T_{d}$, it is proved that $\mathcal{T}$ is an TSFT if and only if $\mathcal{T}$ has the shadowing property (\cite[Theorem 4.55]{binder2019limitPH}, \cite[Theorem 3.6]{bucki2024stability}). The well-known and analogous result of a $\mathbb{N}$ shift is extended by this result. The analogous result of the $\mathbb{Z}^{d}$ (or $\mathbb{N}^{d}$) shifts can be found in \cite{oprocha2008shadowing}. Below is the list of the new definitions of the shadowing property for the tree shift on $T_{M}$.

\begin{definition}
Let $\mathcal{T}$ be a tree-shift defined on $T_{M}$.

\begin{enumerate}
\item A \emph{projected orbit }of $\mathcal{T}$ is a function $\mathcal{O}:T_{M}\rightarrow \mathcal{T}^{\mathcal{I}}$ such that 
\[
\mathcal{O}(g)\in \mathcal{T}^{(\eta)},~\forall g\in T_{M}^{(\eta )},~\forall \eta \in\mathcal{I}.
\]

\item Let $\delta >0$, a projected orbit $\mathcal{O}$ is a $\delta$\emph{-projected pseudo orbit} of $\mathcal{T}$ if 
\[
d(\sigma_{g_{i}}(\mathcal{O}(g)), \mathcal{O}(gg_{i})) <\delta,~\forall g_i\in\Sigma \mbox{ with }gg_i\in T_M.
\]

\item Let $\varepsilon >0$, a projected pseudo orbit $\mathcal{O}$ of $\mathcal{T}$ is $\varepsilon $\emph{-shadowed} by a point $t\in \mathcal{T}$ if $d(\sigma_{g}(t),\mathcal{O}(g))<\varepsilon$ for all $g\in T_{M}$.

\item A tree-shift has the \emph{shadowing property }if for all $\varepsilon >0$ there is a $\delta >0$ such that every $\delta$-projected pseudo orbit $\mathcal{O}$ of $\mathcal{T}$ is $\varepsilon$-shadowed by a point $t\in \mathcal{T}$.
\end{enumerate}
\end{definition}

We characterize the shadowing property for tree-shifts on $T_{M}$ in Theorem \ref{Thm: 2}.

\begin{theorem}
\label{Thm: 2}Let $\mathcal{T}\subseteq \mathcal{A}^{T_{M}}$ be a tree-shift. Then $\mathcal{T}$ is a TSFT if and only if $\mathcal{T}$ has the shadowing property.
\end{theorem}

The asymptotically shadowing property for tree-shift $\mathcal{T}$ on $T_{M}$ is defined.

\begin{definition}
Let $\mathcal{T}$ be a tree-shift on $T_{M}$.

\begin{enumerate}
\item An \emph{asymptotically projected pseudo orbit} $\mathcal{O}$ is a projected pseudo orbit of $\mathcal{T}$ and $\lim_{\left\vert g-\epsilon\right\vert \rightarrow \infty }d(\sigma _{g_{i}}(\mathcal{O}(g)),\mathcal{O}(gg_{i}))=0$ for all $g$, $gg_{i}\in T_{M}$\footnote{That is, for every $\delta >0$, there is an integer $n$ such that $\left\vert g-\epsilon \right\vert >n$ such that $d(\sigma _{g_{i}}(\mathcal{O}(g)),\mathcal{O}(gg_{i}))<\delta $ for all $g$, $gg_{i}\in T_{M}$.}.

\item A $\delta $-projected pseudo orbit $\mathcal{O}$ is called an \emph{asymptotically }$\delta $\emph{-projected pseudo orbit} $\mathcal{O}$ of $\mathcal{T}$ if it is also an asymptotic projected pseudo orbit.

\item A projected pseudo orbit $\mathcal{O}$ of $\mathcal{T}$ is \emph{asymptotically shadowed }by a point $t\in \mathcal{T}$ if for all $\varepsilon >0$, there exists an integer $n\in \mathbb{N}$ such that $d(\sigma_{g}(t),\mathcal{O}(g))<\varepsilon $ for all $g\in T_{M}$ with $%
\left\vert g-\epsilon \right\vert >n$.

\item A tree-shift $\mathcal{T}$ has the \emph{asymptotically }$\delta$\emph{-shadowing property }if every asymptotically $\delta $-projected pseudo orbit of $\mathcal{T}$ is asymptotically shadowed by a point of $\mathcal{T}$.
\end{enumerate}
\end{definition}

We prove that a $m$-step TSFT possesses the asymptotically shadowing property. The analogous result of the tree-shift on $T_{d}$ \cite[Theorem 4.61]{binder2019limitPH} can be extended to those on $T_{M}$ by Theorem \ref{Thm: 3}. 

\begin{theorem}\label{Thm: 3}
Every $m$-step TSFT $\mathcal{T}\subseteq\mathcal{A}^{T_M}$ has the asymptotically $2^{-(m+1)}$-shadowing property.
\end{theorem}

In the remainder of this article, we provide the complete proof for Theorems \ref{Thm: 1} and \ref{Thm: 5} in Section \ref{sec2}, and Theorems \ref{Thm: 2} and \ref{Thm: 3} in Section \ref{sec3}.

\section{Proof of Theorems \ref{Thm: 1} and \ref{Thm: 5}}\label{sec2}
Before we prove the Theorem \ref{Thm: 1}, the following lemma is needed.

\begin{lemma}\label{lemma1.5}
The set $\omega^{M}(t)$ is equal to
\begin{equation}\label{lma eq 0}
\left\{ s\in \bigcup\limits_{\eta \in \mathcal{I}}\mathcal{A}^{\eta }:\forall n>0\text{, }\exists \omega_{n}\in T_{M}\text{ with }\left\vert
\omega_{n}-\epsilon \right\vert >n\ni d(\sigma_{\omega_{n}}(t),s)<n^{-1}\right\} \text{.} 
\end{equation}
\end{lemma}

\begin{proof}
     For any $s\in\omega^M(t)$, by (\ref{3}), we have that for any $n\in \mathbb{N}$,
    \[
    s\in \overline{\left\{\sigma_g(t): g\in  T_M\mbox{ with } |g-\epsilon|>n\right\}}.
    \]
    Then, there exists a sequence $\{w_{n,i}\}_{i=1}^\infty\subseteq T_M$ with $|w_{n,i}-\epsilon|>n$ such that
\begin{equation}\label{lma eq 1}
    \lim_{i\to\infty}d(\sigma_{w_{n,i}}(t),s)=0.
\end{equation}
  By (\ref{lma eq 1}), there exists $m(n)\in \mathbb{N}$ such that if $i\geq m(n)$, then 
    \[
    d\left(\sigma_{w_{n,i}}(t), s\right)<n^{-1}.
    \]
    Let $w_n=w_{n,m(n)}\in T_M$. We have $|w_n-\epsilon|>n$ and 
    \[
    d\left(\sigma_{w_n}(t), s\right)=d\left(\sigma_{w_{n,m(n)}}(t), s\right)<n^{-1}.
    \]
    Thus, $s$ in (\ref{lma eq 0}).

    Conversely, if $s$ in (\ref{lma eq 0}), then $s\in \bigcup_{\eta \in \mathcal{I}}\mathcal{A}^{\eta }$ such that 
    \begin{equation*}
      \forall n>0,\exists w_n\in T_M\mbox{ with } |w_n-\epsilon|>n \ni d(\sigma_{w_n}(t),s)<n^{-1}.  
    \end{equation*}
    Note that for any $n\in\mathbb{N}$, $\{\sigma_{w_{n+i}}(t)\}_{i=1}^\infty\subseteq
    \left\{\sigma_g(t): g\in  T_M\mbox{ with } |g-\epsilon|>n\right\}$ (since $|w_{n+i}-\epsilon|>n+i>n$)
    and 
    \[
    \lim_{i\to\infty} d\left(\sigma_{w_{n+i}}(t), s\right)<\lim_{i\to\infty} (n+i)^{-1}=0.
    \]
    Hence, \[s\in \overline{\left\{\sigma_g(t): g\in  T_M\mbox{ with } |g-\epsilon|>n\right\}},~\forall n\in \mathbb{N}.
    \] 
    Thus, $s\in\omega^M(t)$.
\end{proof}

\begin{proof}[Proof of Theorem \ref{Thm: 1}]  
\item[\bf (1)] We first prove that $\omega^M(t)$ is invariant. For $s\in\omega^M(t)$ and $n>0$, according to Lemma \ref{lemma1.5}, there exists $g\in T_M$ with $|g-\epsilon|>2n$ such that $d(\sigma_{g}(t), s )<(2n)^{-1}$. Then, for any $g_i\in\Sigma$ with $gg_i\in T_M$, we have
\begin{equation*}
d\left(\sigma_{gg_i}(t),\sigma_{g_i}(s)\right)<2\cdot d\left(\sigma_g(t),(s)\right)<2\cdot (2n)^{-1}=n^{-1}
\end{equation*}
Moreover, $|gg_i-\epsilon|=|g-\epsilon|+1>2n+1>n$. By Lemma \ref{lemma1.5} again, $\sigma_{g_i}(s)\in \omega^M(t)$ for all $\sigma_{g_i}(s) \neq \emptyset$ hence $\omega^M(t)$ is invariant. 

The proof of invariance of $\omega_{F_p}^M(t)$ is similar to $\omega^M(t)$, so we omit it here.

\item[\bf (2)] The proof is directly obtained by the definitions of $\omega^M(t)$, $ \omega_p^M(t)$ and $\omega_{F_p}^M(t)$. 

\item[\bf (3)]
Let $T_M$ be a Markov Cayley tree, and $\mathcal{T}$ be a tree-shift on $T_M$ and $t\in \mathcal{T}$. If $\omega_{CPS}^M(t)=\emptyset$, then it is clear that $\omega_{CPS}^M(t)\trianglelefteq\omega^M(t)$. If $\omega_{CPS}^M(t)\neq \emptyset$, then for any $v\in \omega_{CPS}^M(t)$, there exist $1\leq \ell \leq |\mathcal{I}|$ and $\xi_1,...,\xi_{\ell}\in \mathcal{I}$ such that 
\[
v=(t_1,..., t_\ell)\in \omega_{CPS}^{M;[\xi_1:\xi_{\ell}]}(t),
\]
and for any $n>0$, there exists $C_{n}\in\mathcal{C}^{[\xi_{1}:\xi_{\ell}]}$ with $\underline{m}_{C_n} >n$ such that
\begin{equation*}
   \forall g\in C_n,~\mbox{if }F_{T_M}(g)=\xi_i,\mbox{ then }d(\sigma_g(t),t_i)<n^{-1}. 
\end{equation*}
This implies that for all $1\leq i \leq \ell$ and for such fixed $n$, there exists an $w_i\in C_n$ with $F_{T_M}(w_i)=\xi_i$ and $|w_i-\epsilon|>n$ such that 
\[
d(\sigma_{w_i}(t),t_i)<n^{-1}.
\]
By Lemma \ref{lemma1.5}, we have 
\[
t_i\in \omega^M(t),~\forall 1\leq i\leq \ell.
\]
Thus, $\omega_{CPS}^M(t)\trianglelefteq \omega^M(t)$.
\item[\bf (4)] The proof of (\ref{6}) is quickly obtained by (\ref{5}).
\end{proof}

\begin{proof}[Proof of Theorem \ref{Thm: 5}]
\item[\bf (1)]
Let $\bigcup_{\eta\in\mathcal{I}} Y^{(\eta)}\in \overline{\rm PICT}$, we claim that $\bigcup_{\eta\in\mathcal{I}} Y^{(\eta)}\in {\rm PICT}$. For $\varepsilon>0$, choosing $\bigcup_{\eta\in\mathcal{I}}  
    Z^{(\eta)}\in {\rm  PICT}$ with 
    \begin{equation}\label{eq 1}
    d'\left(\bigcup_{\eta\in\mathcal{I}} Y^{(\eta)},\bigcup_{\eta\in\mathcal{I}} Z^{(\eta)}\right)<\frac{\varepsilon}{6},
    \end{equation}
    where $d'(\bigcup_{\eta\in\mathcal{I}}  Y^{(\eta)}, \bigcup_{\eta\in\mathcal{I}} Z^{(\eta)}):=\max\{d(a,b): a\in Y^{(\eta)},b\in Z^{(\eta)}, \eta\in\mathcal{I}\}$ (for simplicity, we write $d'=d$). By (\ref{eq 1}), for any $y, y^\prime\in \bigcup_{\eta\in\mathcal{I}} Y^{(\eta)}$, there exist $z, z^\prime\in \bigcup_{\eta\in\mathcal{I}} Z^{(\eta)}$ such that $d(y,z)<\frac{\varepsilon}{6}$ and $d(y^\prime, z^\prime)<\frac{\varepsilon}{6}$.
    
     Since $\bigcup_{\eta\in\mathcal{I}} Z^{(\eta)}\in {\rm PICT}$, there exists an $\frac{\varepsilon}{2}$-projected chain ($\frac{\varepsilon}{2}$-PC) of $\bigcup_{\eta\in\mathcal{I}} Z^{(\eta)}$ indexed by $g_{i_1}g_{i_2}\cdots g_{i_n}\in T_M$, and denoted by $\{z=z_1, z_2, ..., z_{n+1}=z^\prime\}$, such that $d(\sigma_{g_{i_j}}(z_j), z_{j+1}) < \frac{\varepsilon}{2}$ for all $1\leq j \leq n$. 
     
     By (\ref{eq 1}) again, we can choose $y=y_1, y_2, ..., y_{n+1}=y^\prime\in \bigcup_{\eta\in\mathcal{I}} Y^{(\eta)}$ such that 
    \begin{equation*}
    d\left(y_i, z_i\right)<\frac{\varepsilon}{6}, \forall 1\leq i \leq n+1.
    \end{equation*}   
    Then, for $1\leq j\leq n$, 
    \begin{align*}
        d\left(\sigma_{g_{i_j}}(y_j), y_{j+1}\right)&\leq d\left(\sigma_{g_{i_j}}(y_j),\sigma_{g_{i_j}}(z_j) \right)+d\left(\sigma_{g_{i_j}}(z_j),z_{j+1}\right)+d\left(z_{j+1},y_{j+1}\right)\\
        &<2\cdot d\left(y_j, z_j \right)+\frac{\varepsilon}{2}+\frac{\varepsilon}{6}\\
        &<2\cdot\frac{\varepsilon}{6}+\frac{\varepsilon}{2}+\frac{\varepsilon}{6}=\epsilon.
    \end{align*} 
    Thus, $\{y=y_1, y_2, ..., y_{n+1}=y^\prime\}$ is an $\varepsilon$-PC of $\bigcup_{\eta\in\mathcal{I}} Y^{(\eta)}$ indexed by $g_{i_1}g_{i_2}\cdots g_{i_n}\in T_M$. Hence, $\bigcup_{\eta\in\mathcal{I}} Y^{(\eta)}\in {\rm PICT}$. Therefore, PICT is closed.

\item[\bf (2)]We first claim that $\forall \varepsilon>0,~ p=(\epsilon, g_{i_1}, g_{i_1}g_{i_2}, ...)\in\partial  T_M$ and $t\in \mathcal{T}$, there exists an $N(\varepsilon)>0 $ such that if $n\geq N$ then $d(\sigma_{g_{i_1}\cdots g_{i_n}}(t),\omega_p^M(t))<\varepsilon$. Arguing contrapositively, we have an increasing sequence of integers $\{n_j\}_{j=1}^\infty$ with 
 \[
 d\left(\sigma_{g_{i_1}\cdots g_{i_{n_j}}}(t),\omega_p^M(t)\right) \geq \varepsilon.
 \]
 By passing to a subsequence if necessary, $\{\sigma_{g_{i_1}\cdots g_{i_{n_j}}}(t)\}_{j=1}^\infty$ converges to a point $s\in \omega_p^M(t)$. This is contradictory to $d(s, \omega_p^M(t)) \geq \varepsilon$. This ends the proof of the claim.
 
 Let $\omega_p^M(t)\in \mathcal{W}_p^M$ and $\varepsilon >0$. By the above claim, there exists an $N>0$ such that if $n>N$ then
\begin{equation}\label{eq 1.10-2-1}
    d\left(\sigma_{g_{i_1}\cdots g_{i_n}}(t),\omega_p^M(t)\right)<\frac{\varepsilon}{3}.
\end{equation}

 By (\ref{eq 1.10-2-1}), for any $y,y^\prime \in \omega_p^M(t)$, there are $m,n>N$ such that
\begin{equation*}
d\left(\sigma_{g_{i_1}\cdots g_{i_n}}(t), y\right)<\frac{\varepsilon}{3} \text{ and } d\left(\sigma_{g_{i_1}\cdots g_{i_m}}(t), y^\prime\right)<\frac{\varepsilon}{3}.
\end{equation*}

W.L.O.G, let $n<m$. By (\ref{eq 1.10-2-1}) again, for any $n< j< m$, there is $y_j\in  \omega_p^M(t)$ such that 
 \begin{equation*}
    d\left(\sigma_{g_{i_1}\cdots g_{i_j}}(t),y_j\right)<\frac{\varepsilon}{3}.
\end{equation*}
Denoting $y=y_n$ and $y^\prime=y_m$, we now verify that $\{y=y_n, y_{n+1}, ..., y_{m}=y^\prime\}$ is an $\varepsilon$-PC of $\omega_p^M(t)$ indexed by $g=g_{i_n}\cdots g_{i_m}$. For $n\leq j< m$,
\begin{align*}
    d\left(\sigma_{g_{i_{j+1}}}(y_j),y_{j+1}\right)
    \leq &d\left(\sigma_{g_{i_{j+1}}}(y_j),\sigma_{g_{i_1}\cdots g_{i_{j+1}}}(t)\right)+d\left(\sigma_{g_{i_1}\cdots g_{i_{j+1}}}(t),y_{j+1}\right)\\
    \leq &2\cdot d\left(y_j,\sigma_{g_{i_1}\cdots g_{i_j}}(t)\right)+d\left(\sigma_{g_{i_1}\cdots g_{i_{j+1}}}(t),y_{j+1}\right)\\
    <& 2\cdot \frac{\varepsilon}{3}+\frac{\varepsilon}{3}=\epsilon.
\end{align*}
Thus, $\omega_p^M(t)\in {\rm PICT}$. 
\item[\bf (3)] Let $\mathcal{T}$ be a $m$-step TSFT. By (2) of Theorem \ref{Thm: 5}, we have $\mathcal{W}_p^M\subseteq$ PICT. It remains to show that PICT $\subseteq \mathcal{W}_p^M$. Let $\bigcup_{\eta\in\mathcal{I}} Y^{(\eta)}\in$ PICT, we claim that
\[
\bigcup_{\eta\in\mathcal{I}} Y^{(\eta)}=\omega_p^M(t)\mbox{ for some }t\in\mathcal{T}\mbox{ and }p\in\partial T_M.
\]
For $r\geq m+1$, let $\{x_i^r\}_{i=0}^{n_r}\subseteq \bigcup_{\eta\in\mathcal{I}} Y^{(\eta)}$ be a sequence that $2^{-r}$ covers $\bigcup_{\eta\in\mathcal{I}} Y^{(\eta)}$ (That is, $\forall x\in \bigcup_{\eta\in\mathcal{I}} Y^{(\eta)}$, $\exists 0\leq i\leq n_r \ni d(x,x_i^r)<2^{-r}$). Such $n_r$ is finite since 
\begin{equation*}
    \left|\bigcup_{\eta\in\mathcal{I}}\left\{y|_{\Delta_r^{(\eta)}}: y\in Y^{(\eta)}\right\}\right|\leq
    \sum_{\eta\in\mathcal{I}}\left|\left\{y|_{\Delta_r^{(\eta)}}: y\in Y^{(\eta)}\right\}\right|\leq 
    \sum_{\eta\in\mathcal{I}} |\mathcal{A}|^{|\Delta_r^{(\eta)}|}<\infty.
\end{equation*}

Since $\bigcup_{\eta\in\mathcal{I}} Y^{(\eta)}\in$ PICT, there exists a $2^{-r}$-PC from $x_i^{r}$ to $x_{i+1}^{r}$ indexed by $u_i$ where $u_i$ begins with $x_i^{r}$ and ends with $x_{i+1}^{r}$. By concatenating these chains, we can obtain a $2^{-r}$-PC $\{x_0^{r}=y_0^{r},...,y_{\ell_r}^r=x_{n_r}^r\}$ from $x_0^r$ to $x_{n_r}^r$ indexed by $v_1^r\cdots v_{\ell_r}^r$ and for each $0\leq i\leq n_r$ there exists $0\leq j \leq \ell_r$ such that $x_i^r=y_j^r$.

Let
\begin{equation*}
   \{z_i\}_{i=0}^\infty:= \left\{y_0^{m+1}, ..., y_{\ell_{m+1}}^{m+1}=y_0^{m+2}, y_1^{m+2}, ..., y_{\ell_{m+2}}^{m+2}=y_0^{m+3}, y_1^{m+3}, ...\right\},
\end{equation*}
and let 
\begin{align*}
    p&:=(\epsilon, v_1^{m+1}, v_1^{m+1}v_2^{m+1}, ..., v_1^{m+1}\cdots v_{\ell_{m+1}}^{m+1}, v_1^{m+1}\cdots v_{\ell_{m+1}}^{m+1}v_2^{m+2}, ...)\\
    &:=(\epsilon, g_{i_1}, g_{i_1}g_{i_2}, ...).
\end{align*}
Then, $\{z_i\}_{i=0}^\infty\subseteq \bigcup_{\eta\in\mathcal{I}} Y^{(\eta)}$ and $p\in \partial T_M$. Note that
\begin{equation}\label{01}
    d\left(\sigma_{g_{i_{j+1}}}(z_j),z_{j+1}\right)<2^{-(m+1)},~\forall j\in\mathbb{N},
\end{equation}
and 
\begin{equation}\label{02}
    \forall n>m+1,~\exists k_n \ni d\left(\sigma_{g_{i_{j+1}}}(z_j),z_{j+1}\right)<2^{-n},~\forall j>k_n. 
\end{equation}

Define 
\begin{align*}
    t|_{g}&:=z_0|_{g},~\forall g\in T_M\mbox{ with }g\neq g_{i_1}g',\\
   \sigma_{g_{i_1}\cdots g_{i_n}}(t)|_{\epsilon}&:=z_{n}|_{\epsilon},~\forall n\in \mathbb{N},\\
    \sigma_{g_{i_1}\cdots g_{i_j} g}(t)|_{g}&:=z_{j}|_{g},~\forall g_{i_1}\cdots g_{i_j} g\in T_M\mbox{ with }g\neq g_{i_{j+1}}g'.
\end{align*}
Since for $g=g_{j}g'\in T_M$ with $g_j\neq g_{i_1}$, we have
\begin{equation}\label{03}
    \sigma_g(t)|_{\Delta_{m}^{(F_{T_M}(g))}}=\sigma_g(z_0)|_{\Delta_{m}^{(F_{T_M}(g))}}\notin \mathcal{F}.
\end{equation}
Since for $g=g_{i_1}\cdots g_{i_s}g_jg'\in T_M$ ($s\geq 1$) with $g_j\neq g_{i_{s+1}}$, we have
\begin{equation}\label{04}
    \sigma_g(t)|_{\Delta_m^{(F_{T_M}(g))}}=\sigma_{g_jg'}(z_{i_s})|_{\Delta_m^{(F_{T_M}(g))}}\notin \mathcal{F}.
\end{equation}
Since for $g=\epsilon$ and $u\in \Delta_m^{(T_M)}$, if $u=g_{i_1}\cdots g_{i_\ell}$ ($1\leq \ell\leq m$), then, by (\ref{01}), we have
\begin{equation}\label{06}
    \begin{aligned}
    &d(z_{i_\ell}, \sigma_u(z_0))=d(z_{i_\ell},\sigma_{g_{i_1}\cdots g_{i_\ell}}(z_0))\\
    \leq &d(z_{i_\ell}, \sigma_{g_{i_{\ell}}}(z_{i_{\ell-1}}))+d(\sigma_{g_{i_{\ell}}}(z_{i_{\ell-1}}),\sigma_{g_{i_{\ell}}}(\sigma_{g_{i_{\ell-1}}}(z_{i_{\ell-2}})))  \\
    &+\cdots +d(\sigma_{g_{i_2}\cdots g_{i_\ell}}(z_{i_1}),\sigma_{g_{i_2}\cdots g_{i_\ell}}(\sigma_{g_{i_1}}(z_0)))\\
    \leq &d(z_{i_\ell}, \sigma_{g_{i_{\ell}}}(z_{i_{\ell-1}}))+2\cdot d(z_{i_{\ell-1}},\sigma_{g_{i_{\ell-1}}}(z_{i_{\ell-2}}))  \\
    &+\cdots +2^{\ell-1}\cdot d(z_{i_1},\sigma_{g_{i_1}}(z_0))\\
    <&2^{-(m+1)}(1+\cdots +2^{\ell-1})\leq 2^{-(m+1)}(1+\cdots +2^{m-1})<1.
\end{aligned}
\end{equation}
Thus,
\begin{equation}\label{05}
    \sigma_u(t)|_{\epsilon}=\sigma_{g_{i_1}\cdots g_{i_s}}(t)|_{\epsilon}=z_{i_s}|_{\epsilon}=\sigma_u(z_0)|_{\epsilon}.
\end{equation}
If $u=g_{i_1}\cdots g_{i_s}g_{\ell}g'$ with $g_{\ell}\neq g_{i_{s+1}}$, then, by (\ref{06}), we have
\begin{align*}
    &d(\sigma_{g_\ell g'}(z_{i_s}), \sigma_u(z_0))=d(\sigma_{g_\ell g'}(z_{i_s}),\sigma_{g_{i_1}\cdots g_{i_s}g_{\ell}g'}(z_0))\\
    \leq &2^{|g_\ell g'-\epsilon|}\cdot d(z_{i_s}, \sigma_{g_{i_1}\cdots g_{i_s}}(z_0))\\
    <&2^{|g_\ell g'-\epsilon|}\cdot 2^{-(m+1)}(1+\cdots +2^{s-1})<1.~(\mbox{since }|g_\ell g'-\epsilon|+s\leq m) 
\end{align*}
Thus,
\begin{equation}\label{07}
    \sigma_u(t)|_{\epsilon}=\sigma_{g_{i_1}\cdots g_{i_s}g_{\ell}g'}(t)|_{\epsilon}=\sigma_{g_{\ell}g'}(z_{i_s})|_{\epsilon}=\sigma_u(z_0)|_{\epsilon}.
\end{equation}
Hence by (\ref{05}) and (\ref{07}), we obtain
\begin{equation}\label{012}
    t|_{\Delta_m^{(T_M)}}=z_0|_{\Delta_m^{(T_M)}}\notin \mathcal{F}.
\end{equation}
Since for $g=g_{i_1}\cdots g_{i_s}$ ($s\geq 1$) and $u\in \Delta_m^{(F_{T_M}(g))}$, if $u=g_{i_{s+1}}\cdots g_{i_{s+\ell}}$ ($1\leq \ell \leq m$), then, by (\ref{01}), we have
\begin{equation}\label{08}
    \begin{aligned}
    &d(z_{i_{s+\ell}}, \sigma_u(z_{i_s}))=d(z_{i_{s+\ell}},\sigma_{g_{i_{s+1}}\cdots g_{i_{s+\ell}}}(z_{i_s}))\\
    \leq &d(z_{i_{s+\ell}}, \sigma_{g_{i_{s+\ell}}}(z_{i_{s+\ell-1}}))+d(\sigma_{g_{i_{s+\ell}}}(z_{i_{s+\ell-1}}),\sigma_{g_{i_{s+\ell}}}(\sigma_{g_{i_{s+\ell-1}}}(z_{i_{s+\ell-2}})))  \\
    &+\cdots +d(\sigma_{g_{i_{s+2}}\cdots g_{i_{s+\ell}}}(z_{i_{s+1}}),\sigma_{g_{i_{s+2}}\cdots g_{i_{s+\ell}}}(\sigma_{g_{i_{s+1}}}(z_{i_s})))\\
    \leq &d(z_{i_{s+\ell}}, \sigma_{g_{i_{s+\ell}}}(z_{i_{s+\ell-1}}))+2\cdot d(z_{i_{s+\ell-1}},\sigma_{g_{i_{s+\ell-1}}}(z_{i_{s+\ell-2}}))  \\
    &+\cdots +2^{\ell-1}\cdot d(z_{i_{s+1}},\sigma_{g_{i_{s+1}}}(z_{i_s}))\\
    <&2^{-(m+1)}(1+\cdots +2^{\ell-1})\leq 2^{-(m+1)}(1+\cdots +2^{m-1})<1.
\end{aligned}
\end{equation}
Thus,
\begin{equation}\label{09}
    \sigma_{gu}(t)|_{\epsilon}=\sigma_{g_{i_1}\cdots g_{i_{s+\ell}}}(t)|_{\epsilon}=z_{i_{s+\ell}}|_{\epsilon}=\sigma_u(z_{i_s})|_{\epsilon}.
\end{equation}
If $u=g_{i_{s+1}}\cdots g_{i_{s+\ell}}g_{j}g'$ with $g_{j}\neq g_{i_{s+\ell+1}}$, then, by (\ref{08}), we have
\begin{align*}
    &d(\sigma_{g_j g'}(z_{i_{s+\ell}}), \sigma_u(z_{i_s}))=d(\sigma_{g_j g'}(z_{i_{s+\ell}}),\sigma_{g_{i_{s+1}}\cdots g_{i_{s+\ell}}g_j g'}(z_{i_s}))\\
    \leq &2^{|g_j g'-\epsilon|}\cdot d(z_{i_{s+\ell}}, \sigma_{g_{i_{s+1}}\cdots g_{i_{s+\ell}}}(z_{i_s}))\\
    <&2^{|g_j g'-\epsilon|}\cdot 2^{-(m+1)}(1+\cdots +2^{\ell-1})<1.~(\mbox{since }|g_j g'-\epsilon|+\ell\leq m) 
\end{align*}
Thus,
\begin{equation}\label{010}
    \sigma_{gu}(t)|_{\epsilon}=\sigma_{g_{i_1}\cdots g_{i_{s+\ell}}g_j g'}(t)|_{\epsilon}=\sigma_{g_j g'}(z_{i_{s+\ell}})|_{\epsilon}=\sigma_u(z_{i_s})|_{\epsilon}.
\end{equation}
Hence by (\ref{09}) and (\ref{010}), we obtain
\begin{equation}\label{011}
    t|_{g\Delta_m^{(F_{T_M}(g))}}=z_{i_s}|_{\Delta_m^{(F_{T_M}(g))}}\notin \mathcal{F}.
\end{equation}
Therefore, by (\ref{03}), (\ref{04}), (\ref{012}) and (\ref{011}), we have $t\in \mathcal{T}$. Now, we already construct a $\omega_p^M(t)$ with above $p\in\partial T_M$ and $t\in \mathcal{T}$. It remains to verify that 
\[
\bigcup_{\eta\in\mathcal{I}} Y^{(\eta)}=\omega_p^M(t).
\]

 For $y\in \bigcup_{\eta\in\mathcal{I}} Y^{(\eta)}$ and $n'>m+1$, since $\{x_i^r\}_{i=0}^{n_r}$ is $2^{-r}$ covering of $ \bigcup_{\eta\in\mathcal{I}} Y^{(\eta)}$, we have that there exist $r\geq n'+1$ and $0\leq s\leq n_r$ such that 
\begin{equation}\label{013}
    d\left(y, x_s^r\right)<2^{-(n'+1)},
\end{equation}
and ,by (\ref{02}) and definitions of $t$ and $p$, there exists $j>k_{n'+1}$ such that 
\begin{equation}\label{014}
    \sigma_{g_{i_1}\cdots g_{i_j}}(t)|_{\Delta_{n'+1}^{(F_{T_M}(g_{i_j}))}}=x_s^r|_{\Delta_{n'+1}^{(F_{T_M}(g_{i_j}))}}.
\end{equation}
By (\ref{013}) and (\ref{014}), we have
\begin{equation*}
    d(y, \sigma_{g_{i_1}\cdots g_{i_j}}(t))\leq d(y,x_s^r)+d(x_s^r, \sigma_{g_{i_1}\cdots g_{i_j}}(t))<2^{-(n'+1)}+2^{-(n'+1)}=2^{-n'}.
\end{equation*}
Thus, $y\in \omega_p^M(t)$.

Conversely, for $y\in \omega_p^M(t)$, by the definition of $ \omega_p^M(t)$, we have
\begin{equation*}
    \forall n'>m+1,~ \exists \{y_{n',i}\}_{i=1}^\infty\subseteq \{\sigma_{g_{i_1}\cdots g_{i_j}}(t): j\geq k_{n'}\}\ni\lim_{i\to\infty} d(y,y_{n',i})=0,
\end{equation*}
where $k_{n'}$ defined in (\ref{02}). Then, by (\ref{02}), there exists $y_{n'}\in \{y_{n',i}\}_{i=1}^\infty$ such that  
\begin{equation}\label{015}
    d(y,y_{n'})<2^{-n'},
\end{equation}
where $y_{n'}=\sigma_{g_{i_1}\cdots g_{i_{s(n')}}}(t)$ for some $s(n')\geq k_{n'}$.

Thus, by (\ref{015}),
\begin{equation*}
    y|_{\Delta_{n'}^{(F_{T_M}(g_{i_{s(n')}}))}}=y_{n'}|_{\Delta_{n'}^{(F_{T_M}(g_{i_{s(n')}}))}}= z_{s(n')}|_{\Delta_{n'}^{(F_{T_M}(g_{i_{s(n')}}))}},
\end{equation*}
where $z_{s(n')}\in  \bigcup_{\eta\in\mathcal{I}} Y^{(\eta)}$. Hence, 
\begin{equation*}
    d(y,z_{s(n')})<2^{-n'}.
\end{equation*}
Thus,
\begin{equation}\label{016}
   \lim_{n'\to\infty}  d(y,z_{s(n')})=0.
\end{equation}
Since $z_{s(n')} \in \bigcup_{\eta\in\mathcal{I}} Y^{(\eta)}$ and $\bigcup_{\eta\in\mathcal{I}} Y^{(\eta)}$ is closed, and (\ref{016}), we have $y\in \bigcup_{\eta\in\mathcal{I}} Y^{(\eta)}$. The proof is complete.
\end{proof}

\section{Proof of Theorems \ref{Thm: 2} and \ref{Thm: 3}}\label{sec3}

\begin{proof}[Proof of Theorem \ref{Thm: 2}]
    We first prove that if $\mathcal{T}$ is a TSFT then $\mathcal{T}$ has the shadowing property. Let $\mathcal{T}$ be an $m$-step TSFT with the forbidden set $\mathcal{F}\subseteq \cup_{\eta\in\mathcal{I}}\mathcal{A}^{\Delta_m^{(\eta)}}$. 
  
  For any $\varepsilon>0$, let $\delta=2^{-s}<\min\{2^{-m}, \epsilon\}$ ($s\in\mathbb{Z}$) and let $\mathcal{O}$ be a $\delta$-projected pseudo orbit ($\delta$-PPO) of $\mathcal{T}$. Define $t\in\mathcal{A}^{T_M}$ by
    \[
    t|_g =\mathcal{O}(g)|_{\epsilon},~\forall g\in T_M.
    \]
    We claim that $t\in \mathcal{T}$, that is, no member of $\mathcal{F}$ appears in $t$. For any $g\in  T_M$, since $\mathcal{O}$ is a $\delta$-PPO, we have
\begin{equation*}
        d(\sigma_{g_i}(\mathcal{O}(g)), \mathcal{O}(gg_i))<\delta,~\forall g_i\in\Sigma \text{ with }gg_i\in T_M.
    \end{equation*}
Then, for $gg_{i_1}\cdots g_{i_k}\in T_M$ with $1\leq k\leq s$, we have
\begin{equation*}
   d(\sigma_{g_{i_1}\cdots g_{i_k}}(\mathcal{O}(g)), \mathcal{O}(gg_{i_1}\cdots g_{i_k}))< 2^{k-1}\cdot \delta <2^{-1}.
\end{equation*}
This implies that
\begin{equation*}
    \sigma_{g_{i_1}\cdots g_{i_k}}(\mathcal{O}(g))|_{\epsilon}=\mathcal{O}(gg_{i_1}\cdots g_{i_k})|_{\epsilon}.
\end{equation*}
Thus, for any $u=g_{i_1}\cdots g_{i_k}\in \Delta_s^{(F_{T_M}(g))}$, we have 
\begin{equation*}
    \sigma_u(\mathcal{O}(g))|_{\epsilon}=\mathcal{O}(gu)|_{\epsilon}=: t|_{gu}.
\end{equation*}
Hence,    
    \begin{equation}\label{12}
        t|_{g\Delta_s^{(F_{T_M}(g))}}=\mathcal{O}(g)|_{\Delta_s^{(F_{T_M}(g))}}.
    \end{equation}
Since $m<s$, we have      
    \begin{equation*}
        t|_{g\Delta_m^{(F_{T_M}(g))}}=\mathcal{O}(g)|_{\Delta_m^{(F_{T_M}(g))}}\notin \mathcal{F}.
    \end{equation*}
    Therefore, $t\in \mathcal{T}$. Furthermore, by (\ref{12}), we have that $\forall g\in T_M$,
    \[
    d\left(\sigma_g(t),\mathcal{O}(g)\right)<2^{-s}<\varepsilon.
    \]
    Therefore, $\mathcal{O}$ is $\varepsilon$-shadowed by a point $t\in \mathcal{T}$, and hence $\mathcal{T}$ has the shadowing property.
    
Conversely, arguing contrapositively, if $\mathcal{T}$ is not a TSFT, then we claim that for any $\delta>0$, there is a $\delta$-PPO of $\mathcal{T}$ which can not be $1$-shadowed by any point in $\mathcal{T}$. Note that if $\mathcal{T}$ is not a TSFT, then $\mathcal{T}$ is not a $m$-step TSFT for all $m\geq 0$. Then, for any $m\geq 1$, there exist a positive integer $m'>m+2$ and a member $P\in\mathcal{F}$ with $P\in \cup_{\eta\in\mathcal{I}}\mathcal{A}^{\Delta_{m'}^{(\eta)}}$ such that no member of $\mathcal{F}\cap \cup_{\eta\in\mathcal{I}}\mathcal{A}^{\Delta_{n}^{(\eta)}}$ appears in $P$ for all $0\leq n\leq m'-1$. 
    
Fix a $\delta>0$, choosing $1\leq m\in \mathbb{N}$ such that $\delta> 2^{-m}$. Then, there exist an $m'>m+2$ and a member $P\in\mathcal{F}\cap\cup_{\eta\in\mathcal{I}}\mathcal{A}^{\Delta_{m'}^{(\eta)}}$ as above. Note that $T_M$ is a Markov-Cayley tree, $\mathcal{I}=\{T_M, F_{T_M}(g_1), ..., F_{T_M}(g_d) \}$. Thus, $P\in\mathcal{A}^{\Delta_{m'}^{(\eta)}}$ for some $\eta\in\{T_M,F_{T_M}(g_1), ..., F_{T_M}(g_d)\}$. 

If $\eta=T_M$, then there exist points $t_0 \in \mathcal{T}^{(T_M)}$ and $t_i\in \mathcal{T}^{(F_{T_M}(g_i))}$ ($1\leq i\leq d$) such that 
    \begin{equation*}
        t_0|_{\Delta_{m'-1}^{(\eta)}}=P|_{\Delta_{m'-1}^{(\eta)}}\mbox{ and }t_i|_{\Delta_{m'-1}^{(F_{T_M}(g_i))}}=P|_{g_i\Delta_{m'-1}^{(F_{T_M}(g_i))}}.
    \end{equation*}
Define
    \begin{align*}
        \mathcal{O}(\epsilon)&:=t_0,\\
        \mathcal{O}(g_i)&:=t_i,~(\forall 1\leq i\leq d),\\
        \mathcal{O}(g_ig)&:=\sigma_g(t_i),~(\forall g_ig \in T_M). 
    \end{align*}
Since 
\begin{align*}
    \sigma_{g_i}(\mathcal{O}(\epsilon))|_{\Delta_{m'-2}^{(F_{T_M}(g_i))}}=&\sigma_{g_i}(t_0)|_{\Delta_{m'-2}^{(F_{T_M}(g_i))}}
    =P|_{g_i\Delta_{m'-2}^{(F_{T_M}(g_i))}}\\
=&t_i|_{\Delta_{m'-2}^{(F_{T_M}(g_i))}}
=\mathcal{O}(g_i)|_{\Delta_{m'-2}^{(F_{T_M}(g_i))}},
\end{align*}
we have 
   \begin{equation}\label{13}
        d\left(\sigma_{g_i}(\mathcal{O}(\epsilon)),\mathcal{O}(g_i)\right)\leq 2^{-(m'-2)}<2^{-m}<\delta.
   \end{equation}
Since for $g=g_ig',gg_j\in T_M$,
\begin{align*}
    \sigma_{g_j}(\mathcal{O}(g))&=\sigma_{g_j}(\mathcal{O}(g_ig'))=\sigma_{g_j}(\sigma_{g'}(t_i))\\
    &=\sigma_{g'g_j}(t_i)=\mathcal{O}(g_ig'g_j)=\mathcal{O}(gg_j),
\end{align*}
we have
\begin{equation}\label{14}
    d(\sigma_{g_j}(\mathcal{O}(g)),\mathcal{O}(gg_j))=0<\delta,~ \forall 1\leq i\leq d.
\end{equation} 
Hence by (\ref{13}) and (\ref{14}), $\mathcal{O}$ is a $\delta$-PPO of $\mathcal{T}$. Now, if $\mathcal{O}$ is 1-shadowed by a point $t\in \mathcal{T}$, then
       \begin{equation}\label{19}
           t|_{g}=\mathcal{O}(g)|_{\epsilon},~\forall g\in T_M.
       \end{equation}
        This implies $t|_{\Delta_{m'}^{(\eta)}}=P$. Since $P\in\mathcal{F}$, we have $t\notin \mathcal{T}$, which is contradictory to $t\in \mathcal{T}$. 
        
         If $\eta=F_{T_M}(g_i)$ for some $1\leq i\leq d$, then there exist points $s_0\in \mathcal{T}^{(T_M)}$ and $s_j\in \mathcal{T}^{(F_{T_M}(g_j))}$ ($\forall g_ig_j\in T_M$) such that
\begin{equation*}
    s_0|_{g_i\Delta_{m'-1}^{(F_{T_M}(g_i))}}=P|_{\Delta_{m'-1}^{(F_{T_M}(g_i))}}\mbox { and }s_j|_{\Delta_{m'-1}^{(F_{T_M}(g_j))}}
=P|_{g_j\Delta_{m'-1}^{(F_{T_M}(g_j))}}.
\end{equation*}
Define
\begin{align*}
    \mathcal{O}(g)&:=\sigma_g(s_0),~\mbox{if }g\in T_M\mbox{ and }g\neq g_ig',\\
    \mathcal{O}(g_i)&:=\sigma_{g_i}(s_0),\\
    \mathcal{O}(g_ig_jg)&:=\sigma_{g}(s_j),~\mbox{if }g_ig_jg\in T_M.
\end{align*}
Since
\begin{align*}
\sigma_{g_\ell}(\mathcal{O}(\epsilon))=\sigma_{g_\ell}(s_0)=\mathcal{O}(g_\ell),
\end{align*}
we have 
\begin{equation}\label{15}
    d(\sigma_{g_\ell}(\mathcal{O}(\epsilon)),\mathcal{O}(g_\ell))=0<\delta,~\forall 1\leq \ell\leq d.
\end{equation}
Since for $g, gg_\ell\in T_M$ with $g\neq g_ig'$,
\begin{equation*}
    \sigma_{g_{\ell}}(\mathcal{O}(g))=\sigma_{g_\ell}(\sigma_g(s_0))=\sigma_{gg_\ell}(s_0)=\mathcal{O}(gg_\ell),
\end{equation*}
we have
\begin{equation}\label{16}
    d(  \sigma_{g_{\ell}}(\mathcal{O}(g)),\mathcal{O}(gg_\ell))=0<\delta.
\end{equation}
Since for $g_ig_\ell\in T_M$,
\begin{align*}
    \sigma_{g_\ell}(\mathcal{O}(g_i))|_{\Delta_{m'-2}^{(F_{T_M}(g_\ell))}}&=  \sigma_{g_\ell}(\sigma_{g_i}(s_0))|_{\Delta_{m'-2}^{(F_{T_M}(g_\ell))}}=\sigma_{g_ig_\ell}(s_0))|_{\Delta_{m'-2}^{(F_{T_M}(g_\ell))}}\\
    &=P|_{g_\ell\Delta_{m'-2}^{(F_{T_M}(g_\ell))}}=s_\ell|_{\Delta_{m'-2}^{(F_{T_M}(g_\ell))}}=\mathcal{O}(g_ig_\ell)|_{\Delta_{m'-2}^{(F_{T_M}(g_\ell))}},
\end{align*}
we have
\begin{equation}\label{17}
    d(\sigma_{g_\ell}(\mathcal{O}(g_i)),\mathcal{O}(g_ig_\ell))\leq 2^{-(m'-2)}<2^{-m}<\delta.
\end{equation}
Since for $g_ig_j gg_\ell\in T_M$,
\begin{equation*}
    \sigma_{g_\ell}(\mathcal{O}(g_ig_j g))=  \sigma_{g_\ell}(\sigma_{g}(s_j))=\sigma_{gg_\ell}(s_j)=\mathcal{O}(g_ig_jgg_\ell),
\end{equation*}
we have
\begin{equation}\label{18}
    d(  \sigma_{g_\ell}(\mathcal{O}(g_ig_j g)),\mathcal{O}(g_ig_jgg_\ell))=0<\delta.
\end{equation}

Hence by (\ref{15}), (\ref{16}), (\ref{17}) and (\ref{18}), $\mathcal{O}$ is a $\delta$-PPO $\mathcal{O}$ of $\mathcal{T}$. Now, if $\mathcal{O}$ is 1-shadowed by a point $t\in \mathcal{T}$, then by (\ref{19}), we have $t|_{g_i\Delta_{m'}^{(F_{T_M}(g_i))}}=P\in \mathcal{F}$. Thus, $t\notin \mathcal{T}$. The proof is complete. 
\end{proof}

\begin{proof}[Proof of Theorem \ref{Thm: 3}]
   Let $\mathcal{T}\subseteq \mathcal{A}^{(T_M)}$ be an $m$-step TSFT. For any asymptotic $2^{-(m+1)}$-projected pseudo orbit ($2^{-(m+1)}$-APPO) $\mathcal{O}$ of $\mathcal{T}$, define
   \[
   t|_{g}:=\mathcal{O}(g)|_{\epsilon}, ~\forall g\in  T_M.
   \]
  
   We claim that $t\in \mathcal{T}$ and $\mathcal{O}$ is asymptotically shadowed by $t$. Since $\mathcal{O}$ is $2^{-(m+1)}$-APPO hence $2^{-(m+1)}$-PPO, we have 
   \begin{equation*}
       t|_{g\Delta_{m}^{(F_{T_M}(g))}}=\mathcal{O}(g)|_{\Delta_{m}^{(F_{T_M}(g))}}\notin\mathcal{F},~\forall g\in T_M.
   \end{equation*}
   Since $\mathcal{T}$ is an $m$-step TSFT, we have $t\in \mathcal{T}$. It remains to show that $\mathcal{O}$ is asymptotically shadowed by $t$. For any $\delta>0$, since $\mathcal{O}$ is $2^{-(m+1)}$-APPO hence APPO, there exists $n\in\mathbb{N}$ such that $\forall g, gg_i\in T_M$ with $|g-\epsilon|>n$, 
   \begin{equation}\label{20}
       d(\sigma_{g_i}(\mathcal{O}(g)),\mathcal{O}(gg_i))<\delta\cdot 2^{-1}.
   \end{equation}
    Let $m^\prime\in \mathbb{N}$ such that $2^{-(m'+1)}< \delta\cdot 2^{-1} \leq 2^{-m'}$. Equation (\ref{20}) gives that
   \[\mathcal{O}(g)|_{g'}=\mathcal{O}(gg')|_{\epsilon},~\forall |gg'-g|\leq m',~\forall g, gg'\in T_M\mbox{ with }|g-\epsilon|>n. 
   \]
    Thus, 
   \[
   t|_{g\Delta_{m'}^{(F_{T_M}(g))}}=\mathcal{O}(g)|_{\Delta_{m'}^{(F_{T_M}(g))}},~\forall g\in T_M\mbox{ with }|g-\epsilon|>n. 
   \]
   Hence, 
   \[
   d\left(\sigma_g(t),\mathcal{O}(g)\right)\leq 2^{-m'}<\delta,~\forall g\in T_M\mbox{ with }|g-\epsilon|>n. 
   \]
   The proof is complete.
\end{proof}

\section*{Acknowledgements}
Ban is partially supported by the National Science and Technology Council, ROC (Contract NSTC 111-2115-M-004-005-MY3). Lai is partially supported by the National Science and Technology Council, ROC (Contract NSTC 111-2811-M-004-002-MY2).
% bibliography ---------------------------------------------------
\bibliographystyle{amsplain}
\bibliography{ban}

\end{document}